\documentclass{article}
\usepackage{amsthm}
\usepackage{amsmath}
\usepackage{amssymb}
\usepackage{enumerate}
\usepackage{mathrsfs}
\usepackage{color}
\usepackage{todonotes}
\usepackage{tikz}
\usepackage{pgfplots}
\pgfplotsset{width=8cm,height=4cm,compat=1.9}
\usepackage{algorithmic,algorithm}
\usepackage{multirow}
\usepackage{graphicx}
\usepackage{subfigure}
\usepackage{caption}
\usepackage{makecell}
\definecolor{darkblue}{rgb}{0.0, 0.0, 0.55}
\definecolor{bordeaux}{rgb}{0.34, 0.01, 0.1}
\usepackage[pagebackref,colorlinks,linkcolor=bordeaux,citecolor=darkblue,urlcolor=black,hypertexnames=true]{hyperref}
\usepackage{mathtools}
\usepackage[figuresright]{rotating}
\usepackage{eqparbox}

\newtheorem{theorem}{Theorem}
\newtheorem{lemma}[theorem]{Lemma}
\newtheorem{corollary}[theorem]{Corollary}

\newtheorem{proposition}[theorem]{Proposition}
\newtheorem{definition}[theorem]{Definition}

\newtheorem{example}[theorem]{Example}
\newtheorem{remark}[theorem]{Remark}

\def\Q{{\mathbb{Q}}}
\def\R{{\mathbb{R}}}

\def\N{{\mathbb{N}}}

\def\x{{\mathbf{x}}}

\def\bS{{\mathbf{S}}}

\def\a{{\boldsymbol{\alpha}}}

\def\b{{\boldsymbol{\beta}}}

\def\A{{\mathscr{A}}}
\def\B{{\mathscr{B}}}

\def\cC{{\mathcal{C}}}

\def\int{\hbox{\rm{int}}}

\def\Conv{\hbox{\rm{conv}}}

\def\vec{\hbox{\rm{vec}}}
\def\tr{\hbox{\rm{tr}}}
\def\det{\hbox{\rm{det}}}

\DeclareMathOperator*{\argmin}{arg\,min}

\makeatletter
\newcommand{\subjclass}[2][1991]{%
  \let\@oldtitle\@title%
  \gdef\@title{\@oldtitle\footnotetext{#1 \emph{Mathematics subject classification.} #2}}%
}
\newcommand{\keywords}[1]{%
  \let\@@oldtitle\@title%
  \gdef\@title{\@@oldtitle\footnotetext{\emph{Key words and phrases.} #1.}}%
}
\makeatother

\begin{document}

\title{Weighted Geometric Mean, Minimum Mediated Set, and Optimal Simple Second-Order Cone Representation}

\author{Jie Wang\thanks{wangjie212@amss.ac.cn, Academy of Mathematics and Systems Science, Chinese Academy of Sciences, \href{https://wangjie212.github.io/jiewang}{https://wangjie212.github.io/jiewang}}} 

\date{\today}

\keywords{weighted geometric mean, minimum mediated set, second-order cone representation, polynomial optimization, semidefinite representation}

\subjclass[2020]{90C25,90C22,90C23}

\maketitle

\begin{abstract}
We study optimal \emph{simple second-order cone representations} (a particular subclass of second-order cone representations) for weighted geometric means, which turns out to be closely related to minimum mediated sets. Several lower and upper bounds on the size of optimal simple second-order cone representations are proved. In the case of bivariate weighted geometric means (equivalently, one dimensional mediated sets), we are able to prove the exact size of an optimal simple second-order cone representation and give an algorithm to compute one.
In the genenal case, fast heuristic algorithms and traversal algorithms are proposed to compute an approximately optimal simple second-order cone representation. Finally, applications to polynomial optimization, matrix optimization, and quantum information are provided.
\end{abstract}

\section{Introduction}

This paper is concerned with the following problem:

\vspace{0.5em}
(\uppercase\expandafter{\romannumeral1}) \emph{Given a weighted geometric mean inequality $x_1^{\lambda_1}\cdots x_m^{\lambda_m}\ge x_{m+1}$ with weights $(\lambda_i)_{i=1}^m\in\Q_+^m$, $\sum_{i=1}^m\lambda_i=1$, 
construct an equivalent representation using as few quadratic inequalities (i.e., $x_ix_j\ge x_k^2$) as possible.}
\vspace{0.5em}

The study of Problem (\uppercase\expandafter{\romannumeral1}) is motivated by the fact that any solution to Problem (\uppercase\expandafter{\romannumeral1}) immediately gives rise to a second-order cone representation for the convex set defined by the weighted geometric mean inequality. One advantage of the second-order cone representation is that the related optimization problem can be solved with off-the-shelf second-order cone programming (SOCP) solvers \cite{alizadeh2003second}, e.g., {\tt Mosek} \cite{mosek}, {\tt ECOS} \cite{ecos}. Naturally, to achieve the best efficiency, we wish to obtain such a second-order cone representation of as small size as possible. 

Note that using the so-called ``tower-of-variable'' construction described in \cite[Lecture 3.3]{ben2001lectures}, one is able to give an equivalent second-order cone representation for a weighted geometric mean inequality, whose size is, however, typically far from optimal. As far as the author knows, the problem of finding optimal second-order cone representations for general weighted geometric mean inequalities has not been investigated much in the literature, except for some special cases, which are mentioned below.
Assume that in Problem (\uppercase\expandafter{\romannumeral1}) the weights are written as $\lambda_i=\frac{s_i}{t},s_i\in\N$ for $i=1,\ldots,m$ and $t\in\N$. In \cite{morenko2013p}, Morenko et al. constructed an ``economical'' representation (without any proof concerning optimality) for a weighted geometric mean inequality in the case of $m=3$ and $t=2^l$ for some $l\in\N$ when studying $p$-norm cone programming. In \cite{kian2019minimal}, Kian, Berk and G\"urler dealt with Problem (\uppercase\expandafter{\romannumeral1}) in the case of $t=2^l$ for some $l\in\N$, and proposed a heuristic algorithm for computing a second-order cone representation of small size.

On the other hand, there are plenty of applications of weighted geometric mean inequalities in optimization. For example:
\vspace{0.5em}

{\bf Second-order cone representations for other types of inequalities \cite[Lecture 3.3.1]{ben2001lectures}}. There are other types of inequalities that can be expressed by weighted geometric mean inequalities, and so the second-order cone representations for weighted geometric mean inequalities can immediately lead to second-order cone representations for these inequalities with rational powers:

$\bullet$ $x_1^{\lambda_1}\cdots x_m^{\lambda_m}\ge x_{m+1}$ $\iff$ $x_1^{\lambda_1}\cdots x_m^{\lambda_m}y^{1-\sum_{i=1}^m\lambda_i}\ge x_{m+1}$, $y=1$ ($\lambda_1,\ldots,\lambda_m>0,\sum_{i=1}^m\lambda_i<1$).

$\bullet$ $x^{\lambda}\le y$ $\iff$ $y^{\frac{1}{\lambda}}z^{1-\frac{1}{\lambda}}\ge x$, $z=1$ $(\lambda>1)$. 

$\bullet$ $x^{\lambda}\ge y$ $\iff$ $x^{\lambda}z^{1-\lambda}\ge y$, $z=1$ $(0<\lambda<1)$. 

$\bullet$ $x^{-\lambda}\le y$ $\iff$ $x^{\frac{\lambda}{1+\lambda}}y^{\frac{1}{1+\lambda}}\ge z$, $z=1$ $(\lambda>0)$. 

$\bullet$ $x_1^{-\lambda_1}\cdots x_m^{-\lambda_m}\le y$ $\iff$ $x_1^{\frac{\lambda_1}{1+\sum_{i=1}^m\lambda_i}}\cdots x_m^{\frac{\lambda_m}{1+\sum_{i=1}^m\lambda_i}}y^{\frac{1}{1+\sum_{i=1}^m\lambda_i}}\ge z$, $z=1$ $(\lambda_1,\ldots,\lambda_m>0)$. 
\vspace{0.5em}

{\bf $p$-norm cone programming.}
Let $p\ge1$ be a rational number. The $p$-norm of a vector $\x\in\R^n$ is $\Vert\x\Vert_p\coloneqq\sum_{i=1}^n|x_i|^{\frac{1}{p}}$. The \emph{$p$-norm cone} of dimension $n+1$ is defined by the set
\begin{equation*}
	\{(\x,z)\in\R^{n+1}\mid z\ge\Vert\x\Vert_p\}.
\end{equation*}
Note that the second-order cone is exactly the $p$-norm cone with $p=2$. The $p$-norm cone admits second-order cone representations since the inequality $z\ge\Vert\x\Vert_p$ can be rewritten as $z\ge\sum_{i=1}^n|x_i|^p/z^{p-1}$, which is equivalent to (\cite[Lecture 3.3.1]{ben2001lectures})
\begin{equation*}
	\exists(y_i)_{i=1}^n,(w_i)_{i=1}^n\in\R_+^n\text{ s.t. }\begin{cases}
		z^{1-\frac{1}{p}}y_i^{\frac{1}{p}}\ge w_i, |x_i|\le w_i, i=1,\ldots,n,\\
		\sum_{i=1}^ny_i=z.
	\end{cases}
\end{equation*}
\vspace{0.5em}

{\bf Power cone programming \cite[Lecture 3.3.1]{ben2001lectures}}.
Given $\lambda_1,\ldots,\lambda_m>0$ with $\sum_{i=1}^m\lambda_i=1$, the corresponding \emph{power cone} is defined by the set
\begin{equation*}
	\{(\x,z)\in\R_+^n\times\R\mid x_1^{\lambda_1}\cdots x_m^{\lambda_m}\ge |z|\}.
\end{equation*}
The power cone admits second-order cone representations since the inequality $x_1^{\lambda_1}\cdots x_m^{\lambda_m}\ge |z|$ is equivalent to
\begin{equation*}
	\exists y\ge0\text{ s.t. }x_1^{\lambda_1}\cdots x_m^{\lambda_m}\ge y, |z|\le y.
\end{equation*}
\vspace{0.5em}

{\bf Matrix optimization.} The second-order cone representation for weighted geometric mean inequalities allows one to give (approximate) semidefinite representations for many matrix functions, as revealed in \cite{fawzi2017lieb,fawzi2019semidefinite,sagnol2013}. These functions include the matrix power function \cite{helton2015free}, Lieb's function, the Tsallis entropy, the Tsallis relative entropy \cite{fawzi2017lieb}, the matrix logarithm function \cite{fawzi2019semidefinite}, which have various applications in matrix optimization and quantum information.

\vspace{0.5em}

A relevant problem to Problem (\uppercase\expandafter{\romannumeral1}) is

\vspace{0.5em}
(\uppercase\expandafter{\romannumeral2}) \emph{Given a set of points $\{\a_{i}\}_{i=1}^m\subseteq\R^{m-1}$ forming the vertices of a simplex and a point $\a_{m+1}=\sum_{i=1}^{m}\lambda_i\a_i$ with $(\lambda_i)_{i=1}^m\in\Q_+^m$ and $\sum_{i=1}^m\lambda_i=1$, find as few points as possible (say, $\{\a_{i}\}_{i=m+2}^{m+n}$) such that every point in $\{\a_{i}\}_{i=m+1}^{m+n}$ is an average of two distinct points in $\{\a_{i}\}_{i=1}^{m+n}$.}
\vspace{0.5em}

Problem (\uppercase\expandafter{\romannumeral2}) arises from the study of nonnegative circuit polynomials \cite{magron2023sonc,soncsocp}. Sums of nonnegative circuit polynomials (SONC) were proposed by Iliman and De Wolff as certificates of polynomial nonnegativity \cite{iw}, and have been employed to solve sparse polynomial optimization problem in a ``degree-free'' manner \cite{diw,dkw,magron2023sonc}. The set of points $\{\a_{i}\}_{i=m+1}^{m+n}$ considered in Problem (\uppercase\expandafter{\romannumeral2}) is called a (minimum) mediated set. Mediated sets over integers were initially introduced by Reznick in \cite{re} to study agiforms, and were extended to the rational case by Wang and Magron \cite{soncsocp}. It was proved in \cite{magron2023sonc} that a circuit polynomial is nonnegative if and only if it can be written as a sum of binomial squares supported on a mediated set, with the number of binomial squares equaling the number of points contained in the mediated set. In \cite{blekherman2022moments}, the notion of mediated sets plays an important role in the study of the tropicalization of pseudo-moment cones and discrete mid-point convexity. More recent research on mediated sets can be found in \cite{hartzer2022initial,powers2021note}.

Interestingly, it turns out that solutions to Problem (\uppercase\expandafter{\romannumeral1}) are in a one-to-one correspondence to solutions to Problem (\uppercase\expandafter{\romannumeral2}). As one would see later, the development of both theory and algorithms in the present paper deeply relies on this appealing connection.

We should point out that in the literature, there is a more general notion of second-order cone representations (also called second-order cone lifts) which are of the form $\pi(K^n\cap L)$ with $K$ being the three-dimensional rotated second-order cone, $L$ being an affine space, and $\pi$ being a linear map \cite{fawzi2018lower,fawzi2022lifting,gouveia2013lifts}. However in this paper, we would rather work on the more restrictive case (called \emph{simple} second-order cone representations) considered in Problem (\uppercase\expandafter{\romannumeral1}) where $\pi$ specializes to a coordinate projection and $L$ is an affine space that constrains certain coordinates to be equal due to the following reasons: (1) representations of this restrictive form can be lifted to semidefinite representations of matrix geometric means in a direct way (see Theorem \ref{sec6:thm0} for a more formal statement about this); (2) representations of this restrictive form admit a one-to-one correspondence to mediated sets which have a nice geometric description; (3) we are able to propose efficient algorithms to compute simple second-order cone representations of (approximately) optimal size.
 
Our main contributions are summarized as follows:

$\bullet$ In Section \ref{sec2}, we prove several lower bounds and upper bounds on the size of optimal simple second-order cone representations for weighted geometric mean inequalities (equivalently, the size of minimum mediated sets). In particular, the lower bounds are shown to be attainable by examples.

$\bullet$ In Section \ref{sec3}, we prove the exact size of an optimal simple second-order cone representation for bivariate weighted geometric mean inequalities, which resolves a conjecture proposed in \cite{magron2023sonc}.
In addition, we provide a binary tree representation of one dimensional ``successive'' minimum mediated sets.

$\bullet$ In Section \ref{sec4}, we propose several heuristic algorithms for computing an approximately optimal simple second-order cone representation of a weighted geometric mean inequality (equivalently, an approximately minimum mediated set) and compare their practical performance. We also propose a brute-force algorithm for computing an exact optimal simple second-order cone representation of a weighted geometric mean inequality (equivalently, a minimum mediated set). Numerical experiments demonstrate that the heuristic algorithms can efficiently produce simple second-order cone representations of size being equal or close to the optimal one.

$\bullet$ In Section \ref{sec5}, we provide applications of the proposed algorithms to polynomial optimization, matrix optimization and quantum information, and demonstrate their efficiency by numerical experiments.

\section{Preliminaries}
Let $\N$, $\R$, $\Q$ be the set of nonnegative integers, real numbers, rational numbers, respectively. Let $\N^*$, $\R_{\ge0},\R_+,\Q_+$ be the set of positive integers, nonnegative real numbers, positive real numbers, positive rational numbers, respectively. For $a\in\R$, $\lceil a\rceil$ stands for the least integer that is no less than $a$. For a tuple of integers $(s_1,\ldots,s_m)\in\N^m$, we use $\hat{s}$ to denote the sum of its entries, i.e., $\hat{s}\coloneqq\sum_{i=1}^ms_i$. For a set $A$, we use $|A|$ to denote the cardinality of $A$. For a finite set $\A\subseteq\N^n$, we denote by $\Conv(\A)$ the convex hull of $\A$, and use $\Conv(\A)^{\circ}$ to denote the relative interior of $\Conv(\A)$. For integers $s_1,s_2,\ldots$, we use the parenthesis $(s_1,s_2,\ldots)$ to denote their greatest common divisor. Let $\bS^n$, $\bS_+^n$, $\bS_{++}^n$ be the set of symmetric matrices, positive semidefinite, positive definite matrices of size $n$, respectively.

The $n$-dimensional \emph{rotated second-order cone} is defined by the set
\begin{equation*}
	K_n\coloneqq\left\{(a_i)_{i=1}^n\in\R^{n}\middle\vert2a_1a_2\ge\sum_{i=3}^{n}a_i^2,a_1\ge0,a_2\ge0\right\}.
\end{equation*}
In this paper, we are mostly interested in the $3$-dimensional rotated second-order cone\footnote{Actually, any $n$-dimensional second-order cone is representable using the $3$-dimensional second-order cone.} ($K\coloneqq K_3$) which we simply refer to as the second-order cone.

\section{Optimal simple second-order cone representations and minimum mediated sets}\label{sec2}
In this paper, a \emph{weighted geometric mean inequality} refers to an inequality of form
\[x_1^{\lambda_1}\cdots x_m^{\lambda_m}\ge x_{m+1}, \text{ with }(\lambda_i)_{i=1}^m\in\R_+^m\text{ and }\sum_{i=1}^m\lambda_i=1,\] 
where the variables $x_1,\ldots,x_{m+1}$ are assumed to be nonnegative. Throughout the paper, we assume that the weights $(\lambda_i)_{i=1}^m\in\Q_+^m$ are rational numbers. It is known that when the weights are rational, the weighted geometric mean inequality is second-order cone representable \cite[Lecture 3.3]{ben2001lectures}. As $(\lambda_i)_{i=1}^m$ are rational numbers by assumption, we can write $\lambda_i=\frac{s_i}{\hat{s}}$, $s_i\in\N$ for $i=1,\ldots,m$. In view of this, we also call $x_1^{s_1}\cdots x_m^{s_m}\ge x_{m+1}^{\hat{s}}$ with $(s_i)_{i=1}^m\in\N^m,(s_1,\ldots,s_m)=1$ a weighted geometric mean inequality.

\begin{definition}
A \emph{simple second-order cone representation} for a weighted geometric mean inequality $x_1^{s_1}\cdots x_m^{s_m}\ge x_{m+1}^{\hat{s}}$ consists of a set of quadratic inequalities $x_{i_k}x_{j_k}\ge x_{m+k}^2,k=1,\ldots,n$ such that 
\begin{equation}\label{sec3:eq1}
x_1^{s_1}\cdots x_m^{s_m}\ge x_{m+1}^{\hat{s}}\iff
\begin{cases}
   x_{i_1}x_{j_1}\ge x_{m+1}^2,\\
   \quad\quad\quad\vdots\\
   x_{i_n}x_{j_n}\ge x_{m+n}^2,\\
   i_k,j_k\in\{1,2,\ldots,m+n\},\quad k=1,\ldots,n,\\
\end{cases}
\end{equation}
where $x_{m+2},\ldots,x_{m+n}$ are $n-1$ auxiliary nonnegative variables. We call $n$ the size of the second-order cone representation. A simple second-order cone representation of minimum size is called an \emph{optimal} simple second-order cone representation whose size is denoted by $L(s_1,\ldots,s_m)$.
\end{definition}

\begin{remark}
The second-order cone representation in \eqref{sec3:eq1} is uniquely determined by the set of integer triples $\{(i_k,j_k,m+k)\}_{k=1}^n$ which is hereafter referred to as a \emph{configuration}.
\end{remark}

\begin{remark}
The value of $L(s_1,\ldots,s_m)$ clearly does not depend on the ordering of $s_1,\ldots,s_m$.
\end{remark}

\begin{example}
The weighted geometric mean inequality $x_1^3x_2^8\ge x_3^{11}$ admits a simple second-order cone representation: $x_2x_6\ge x_3^2,x_1x_3\ge x_4^2,x_3x_4\ge x_5^2,x_4x_5\ge x_6^2$.
\end{example}

A set of points $\A=\{\a_1,\ldots,\a_{m}\}\subseteq\R^{m-1}$ is called a \emph{trellis} if they are affinely independent.
Given a trellis $\A$, we say that a set of points $\B$ is an \emph{$\A$-mediated set}, if every point $\b\in\B$ is an average of two distinct points in $\A\cup\B$. One could immediately deduce that the points in an $\A$-mediated set belong to the convex hull of $\A$. An $\A$-mediated set containing a given point $\b\in\Conv(\A)^{\circ}$ is called an $(\A,\b)$-mediated set and a \emph{minimum $(\A,\b)$-mediated set} is an $(\A,\b)$-mediated set with the smallest cardinality. We sometimes omit the prefixes $\A
,\b$ and simply say (minimum) mediated sets if there is no need to mention the specific $\A
,\b$. If $\b\in\Conv(\A)^{\circ}$, then there exists a unique tuple of scalars $\lambda_1,\ldots,\lambda_{m}\in\R_+$ (the barycentric coordinates) such that $\sum_{i=1}^m\lambda_{i}=1$ and $\b=\sum_{i=1}^{m}\lambda_i\a_i$. It is easy to see that for an $(\A,\b)$-mediated set to exist, the $\lambda_i$'s are necessarily rational numbers.

\begin{example}\label{sec3:ex1}
Let $\A=\{\a_1=(4,2),\a_2=(2,4),\a_3=(0,0)\}$, and $\b_1=(2,2),\b_2=(1,2),\b_3=(3,2)$. It is easy to check by hand that $\{\b_1,\b_2,\b_3\}$ is an $\A$-mediated set; see Figure \ref{sec3:fg1}.

\begin{figure}
\centering
\begin{tikzpicture}
\draw (0,0)--(2,4);
\draw (0,0)--(4,2);
\draw (2,4)--(4,2);
\draw (1,2)--(4,2);
\fill (0,0) circle (2pt);
\node[below left] (1) at (0,0) {$\a_3$};
\fill (2,4) circle (2pt);
\node[above left] (2) at (2,4) {$\a_2$};
\fill (4,2) circle (2pt);
\node[below right] (3) at (4,2) {$\a_1$};
\fill (2,2) circle (2pt);
\node[below] (4) at (2,2) {$\b_1$};
\fill (1,2) circle (2pt);
\node[below left] (5) at (1,2) {$\b_2$};
\fill (3,2) circle (2pt);
\node[below] (6) at (3,2) {$\b_3$};
\end{tikzpicture}
\caption{$\{\b_1,\b_2,\b_3\}$ forms an $\{\a_1,\a_2,\a_3\}$-mediated set.}
\label{sec3:fg1}
\end{figure}

\end{example}

Assume that $\A=\{\a_1,\ldots,\a_{m}\}$ is a trellis and $\a_{m+1}=\sum_{i=1}^{m}\frac{s_i}{\hat{s}}\a_i$ with $(s_i)_{i=1}^m\in(\N^*)^m,(s_1,\ldots,s_m)=1$. If $\{\a_{m+1},\ldots,\a_{m+n}\}$ is an $\A$-mediated set, then by definition  the following system of equations must be satisfied for appropriate subscripts $i_k,j_k$:
\begin{equation}\label{sec3:eq2}
\begin{cases}
   \a_{i_1}+\a_{j_1}=2\a_{m+1},\\
   \quad\quad\quad\vdots\\
   \a_{i_n}+\a_{j_n}=2\a_{m+n},\\
   i_k,j_k\in\{1,2,\ldots,m+n\},\quad k=1,\ldots,n.\\
\end{cases}
\end{equation}
Comparing \eqref{sec3:eq1} with \eqref{sec3:eq2}, we see that there is a one-to-one correspondence between simple second-order cone representations for the weighted geometric mean inequality $x_1^{s_1}\cdots x_m^{s_m}\ge x_{m+1}^{\hat{s}}$ and $\A$-mediated sets containing the point $\a_{m+1}=\sum_{i=1}^{m}\frac{s_i}{\hat{s}}\a_i$:{\bf
\begin{equation*}
\text{simple second-order cone representations for } x_1^{s_1}\cdots x_m^{s_m}\ge x_{m+1}^{\hat{s}}
\end{equation*}
\begin{equation*}
\updownarrow
\end{equation*}
\begin{equation*}
\text{configurations } \{(i_k,j_k,m+k)\}_{k=1}^n
\end{equation*}
\begin{equation*}
\updownarrow
\end{equation*}
\begin{equation*}
\text{$\A$-mediated sets containing $\a_{m+1}=\sum_{i=1}^{m}\frac{s_i}{\hat{s}}\a_i$}
\end{equation*}}
It then follows that $L(s_1,\ldots,s_m)$ also denotes the cardinality of a minimum $(\A,\a_{m+1})$-mediated set.

We now prove a lower bound on the size of optimal simple second-order cone representations for a weighted geometric mean inequality in terms of the number of variables involved.
\begin{theorem}\label{sec3:thm1}
Let $s_1,\ldots,s_m\in\N^*$ be a tuple of integers with $(s_1,\ldots,s_m)=1$. Then 
\begin{equation}
    L(s_1,\ldots,s_m)\ge m-1.
\end{equation}
\end{theorem}
\begin{proof}
Suppose that $\A=\{\a_1,\ldots,\a_{m}\}$ is a trellis and $\a_{m+1},\ldots,\a_{m+n}$ is a minimum $(\A,\a_{m+1})$-mediated set such that
$\sum_{i=1}^{m}s_i\a_i=\hat{s}\a_{m+1}$. Let us consider the system of equations \eqref{sec3:eq2}.
Note first that $\{1,\ldots,m\}\subseteq\cup_{k=1}^n\{i_k,j_k\}$. Moreover, we must have $\{m+2,\ldots,m+n\}\subseteq\cup_{k=1}^n\{i_k,j_k\}$ as if $k\notin\cup_{k=1}^n\{i_k,j_k\}$ for some $k\in\{m+2,\ldots,m+n\}$, then we can delete $\a_{k}$ to obtain a smaller $(\A,\a_{m+1})$-mediated set. From these facts, we get
\begin{equation*}
    2n\ge m+n-1,
\end{equation*}
which yields the desired inequality.
\end{proof}

The lower bound in Theorem \ref{sec3:thm1} is attainable as the following example shows.
\begin{example}
The weighted geometric mean inequality $x_1x_2x_3x_4\ge x_5^{4}$ admits a simple second-order cone representation: $x_6x_7\ge x_5^2,x_1x_2\ge x_6^2,x_3x_4\ge x_7^2$.
\end{example}

In the language of mediated sets, Theorem \ref{sec3:thm1} says that the cardinality of a minimum mediated set is bounded by its dimension (i.e., the dimension of its associated trellis) from below.

From the point of view of mediated sets, we are actually able to give another lower bound on $L(s_1,\ldots,s_m)$ in terms of the sum $\hat{s}$. For the proof, we need the notion of M-matrices and two preliminary results.
\begin{definition}
Let $C=[c_{ij}]$ be a real matrix. Then $C$ is called an \emph{M-matrix} if it satisfies
\begin{enumerate}[(1)]
	\item $c_{ij}\le0$ if $i\ne j$;
	\item $C=tI-B$, where $B$ is a matrix with nonnegative entries, $I$ is an identity matrix, and $t$ is no less than the spectral radius (the maximum of the moduli of the eigenvalues) of $B$.
\end{enumerate}
\end{definition}
A remarkable property of M-matrices is that the determinant of an M-matrix is bounded by the product of its diagonals from above.
\begin{lemma}\label{sec3:lm2}
Let $C=[c_{ij}]\in\R^{n\times n}$ be an M-matrix. Then, it holds
\begin{equation}
	\det(C)\le\prod_{i=1}^nc_{ii}.
\end{equation}
\end{lemma}
\begin{proof}
See e.g., Corollary 4.1.2 of \cite{ando1980}.
\end{proof}

The following lemma adapted from \cite[Lemma 4.3]{re} asserts that a particular class of matrices are M-matrices.
\begin{lemma}\label{sec3:lm1}
Let $C=[c_{ij}]$ be a real matrix such that $c_{ii}=2$ and $c_{ij}\in\{0,-1\}$ if $i\ne j$. Assume that each row of $C$ has at most two $-1$'s and there is no principal submatrix of $C$ in which each row has exactly two $-1$'s. Then $C$ is a non-singular M-matrix.
\end{lemma}
\begin{proof}
Let us write $C=2I-B$, where $I$ is an identity matrix, and $B=[b_{ij}]$ is a matrix with entries being $0$ or $1$. We finish the proof by showing that the spectral radius of $B$ is less than $2$.

Let $\lambda$ be any eigenvalue of $B$ and $\a=(\alpha_i)_i$ a corresponding eigenvector such that $B\a=\lambda\a$. Let $\xi=\max_i\{|\alpha_i|\}$ and $I=\{i\mid |\alpha_i|=\xi\}$. For $i\in I$, let $T(i)=\{j\mid b_{ij}=1\}$. Then for each $i\in I$, we have
\begin{equation*}
|\lambda|\xi=|\lambda\alpha_i|=\Bigl|\sum_{j}b_{ij}\alpha_j\Bigr|=\Bigl|\sum_{j\in T(i)}\alpha_j\Bigr|\le 2\xi.
\end{equation*}
So $|\lambda|\le2$ for $\xi>0$. If $|\lambda|=2$, then for any $i\in I$, we have $|T(i)|=2$ and $|\alpha_j|=\xi$ for $j\in T(i)$ which implies $T(i)\subseteq I$. 
It follows that the principal submatrix of $C$ indexed by $I$ 
has exactly two $-1$'s in each row, which is impossible. 
Thus $|\lambda|<2$ as desired.
\end{proof}

Now we are ready to prove the promised result.
\begin{theorem}\label{sec3:thm2}
Let $\A=\{\a_1,\ldots,\a_{m}\}\subseteq\R^{m-1}$ be a trellis and $\b=\sum_{i=1}^{m}\frac{s_i}{\hat{s}}\a_i$ with $(s_i)_{i=1}^m\in(\N^*)^m,(s_1,\ldots,s_m)=1$. Then for any $(\A,\b)$-mediated set $\B$, one has $|\B|\ge\left\lceil\log_2(\hat{s})\right\rceil$.
\end{theorem}
\begin{proof}
Assume that $\B=\{\b_1,\ldots,\b_n\}$ with $\b_1=\b$ is an $(\A,\b)$-mediated set. By definition, for each $\b_i$, we have one of the following three equations holds for suitable subscripts $j,k\in\{1,\ldots,m\}$:
\begin{subequations}\label{mse}
\begin{align}
    2\b_i&=\a_j+\a_k,\\
    2\b_i-\b_j&=\a_k,\\
    2\b_i-\b_j-\b_k&=\mathbf{0}.\label{mse3}
\end{align}
\end{subequations}
Let us denote the coefficient matrix of \eqref{mse} by $C$ and let $B=[\b_1,\ldots,\b_n]^{\intercal}$ with $\b_i$'s being viewed as column vectors, so that we can rewrite \eqref{mse} in matrix form as $CB=A$, where $A$ is a matrix whose row vectors belong to the subspace generated by $\A$. It is not hard to see that $C$ satisfies the conditions of Lemma \ref{sec3:lm1}; only the assumption that there is no principal submatrix with exactly two $-1$'s in each row is not obviously fulfilled.
Suppose that $C$ has a principal submatrix indexed by $I$ with two $-1$'s in each row and let $\B'=\{\b_i\}_{i\in I}$. Then by construction, every point $\b_i$ in $\B'$ satisfies \eqref{mse3} with $j,k\in I$ and so is an average of two other points in $\B'$, which is however impossible for a finite set. Hence $C$ has no such principal submatrix.
Applying Lemma \ref{sec3:lm1}, we deduce that $C$ is M-matrix. In addition, by Lemma \ref{sec3:lm2}, $\det(C)\le 2^{n}$. Solve $CB=A$ for $\b$ and we obtain $\b=\frac{\sum_{i=1}^{m}r_i\a_i}{\det(C)}$ for some $r_i\in\N^*$, $i=1,\ldots,m$. It follows $\hat{s}\le\det(C)\le2^{n}$ which implies $|\B|=n\ge\left\lceil\log_2(\hat{s})\right\rceil$.
\end{proof}

From Theorem \ref{sec3:thm2}, we immediately obtain the following corollary. 
\begin{corollary}\label{sec3:thm3}
Let $s_1,\ldots,s_m\in\N^*$ be a tuple of integers with $(s_1,\ldots,s_m)=1$. Then 
\begin{equation}
    L(s_1,\ldots,s_m)\ge\left\lceil\log_2\left(\hat{s}\right)\right\rceil.
\end{equation}
\end{corollary}

The lower bound in Corollary \ref{sec3:thm3} is attainable as the following example shows.
\begin{example}
The weighted geometric mean inequality $x_1x_2^2x_3^3\ge x_4^{6}$ admits a simple second-order cone representation: $x_3x_5\ge x_4^2,x_2x_6\ge x_5^2,x_1x_5\ge x_6^2$.
\end{example}

\begin{remark}\label{sec3:rm1}
Combining Theorem \ref{sec3:thm1} with Corollary \ref{sec3:thm3}, we get
\begin{equation}\label{sec3:eq3}
	L(s_1,\ldots,s_m)\ge\max\,\{\left\lceil\log_2\left(\hat{s}\right)\right\rceil, m-1\}.
\end{equation}
\end{remark}

\begin{remark}
An $\A$-mediated set $\B$ is said to be \emph{isomorphic} to an $\A'$-mediated set $\B'$ if there are one-to-one maps $\A\rightarrow\A'$ and $\B\rightarrow\B'$ such that the average relationships \eqref{mse} are preserved under these maps. 

Suppose that $\A=\{\a_1,\ldots,\a_{m}\}\subseteq\R^{m-1}$ is a trellis and $\b=\sum_{i=1}^{m}\frac{s_i}{\hat{s}}\a_i$ with $(s_i)_{i=1}^m\in(\N^*)^m,(s_1,\ldots,s_m)=1$. We may define
\begin{equation*}
	\b'=(s_1,\ldots,s_{m-1})\in\R^{m-1},
\end{equation*}
and
\begin{equation*}
	\a_i'=\hat{s}\mathbf{e}_i\in\R^{m-1},i=1,\ldots,m-1,\a_{m}'=\mathbf{0}\in\R^{m-1},
\end{equation*}
so that
\begin{equation*}
	\b'=\sum_{i=1}^{m}\frac{s_i}{\hat{s}}\a_i'.
\end{equation*}
Here, $(\mathbf{e}_i)_{i=1}^{m-1}$ denotes the standard basis of $\R^{m-1}$.
In addition, let $\A'=\{\a_1',\ldots,\a_{m}'\}$. It can be seen that any $(\A,\b)$-mediated set is isomorphic to an $(\A',\b')$-mediated set and vise verse. Therefore, to study $(\A,\b)$-mediated sets, there is no loss of generality in assuming that the trellis $\A$ comprises the vertices of the standard simplex
and $\b$ is a lattice point lying in the relative interior of this simplex.
\end{remark}

\subsection{More general second-order cone representations}
In this subsection, we consider more general second-order cone representations for a weighted geometric mean inequality and derive lower bounds on the size of such second-order cone representations. Let $K$ be the $3$-dimensional rotated second-order cone. Given a convex cone $C$, we say that $C$ admits a $K^l$-lift if $C=\pi(K^l\cap L)$ where $L$ is an affine space and $\pi$ is a linear map \cite{fawzi2022lifting}.

\begin{theorem}
Let $s_1,\ldots,s_m\in\N^*$ be a tuple of integers with $(s_1,\ldots,s_m)=1$. The size of any second-order cone representation for the weighted geometric mean inequality $x_1^{s_1}\cdots x_m^{s_m}\ge x_{m+1}^{\hat{s}}$ is bounded by $\frac{m}{2}$ from below.
\end{theorem}
\begin{proof}
Let us denote $S\coloneqq\left\{(x_i)_{i=1}^{m+1}\in\R_{\ge0}^{m+1}\mid x_1^{s_1}\cdots x_m^{s_m}\ge x_{m+1}^{\hat{s}}\right\}$. Note that $S$ contains the nonnegative orthant $\R_{\ge0}^m$ as a linear slice (with $x_{m+1}=0$). Therefore if $S$ admits a $K^l$-lift, then so does $\R_{\ge0}^m$. Now since cone lifts induce order embeddings of face posets (\cite[Section 5.1]{fawzi2022lifting}) and the longest chains of non-empty faces of $\R_{\ge0}^m$ and $K^l$ are respectively of length $m+1,2l+1$, we deduce $m+1\le 2l+1$, equivalently, $l\ge\frac{m}{2}$ as desired. 
\end{proof}

\section{A binary tree representation of successive minimum mediated sequences}\label{sec3}
In this section, we focus particularly on the case of one dimensional mediated sets (equivalently, the case of bivariate weighted geometric mean inequalities). For the sake of conciseness, we use the terminology \emph{mediated sequences} to refer to one dimensional mediated sets. More concretely, given an integer $p>0$, a set of integers $A\subseteq\N$ is a \emph{$p$-mediated sequence}, if every number in $A$ is an average of two distinct numbers in $A\cup\{0,p\}$. A $p$-mediated sequence containing a given number $q$ is called a \emph{$(p,q)$-mediated sequence}. As being a mediated sequence is not changed by a scaling, there is no loss of generality in assuming $(p,q)=1$. A \emph{minimum} $(p,q)$-mediated sequence is a $(p,q)$-mediated sequence with the smallest cardinality. We sometimes omit the prefixes and simply say (minimum) mediated sequences if there is no need to mention the specific $p,q$.
\begin{example}
The set $A=\{2,4,5,8\}$ is a minimum $(11,2)$-mediated sequence.
\end{example}
Given $p,q\in\N$ with $0<q<p$, $(p,q)=1$, there is an algorithm (Algorithm \ref{sec4:alg1}) for computing a minimum $(p,q)$-mediated sequence.

\begin{algorithm}
	\renewcommand{\algorithmicrequire}{\textbf{Input:}}
	\renewcommand{\algorithmicensure}{\textbf{Output:}}
	\caption{}\label{sec4:alg1}
	\begin{algorithmic}[1]
		\REQUIRE
		Two integers $p,q$ satisfying $0<q<p$ and $(p,q)=1$
		\ENSURE
		A minimum $(p,q)$-mediated sequence
		\STATE $l\leftarrow\left\lceil\log_2\left(p\right)\right\rceil$;
		\STATE $s_1\leftarrow q$, $s_2\leftarrow p-q$, $s_3\leftarrow 2^l-p$;
		\STATE $t_1\leftarrow p$, $t_2\leftarrow 0$, $t_3\leftarrow q$;
	    \FOR{$k\leftarrow 1$ to $l$}
	    \STATE Find $1\le i\ne j\le 3$ such that $s_i\le s_j$ are odd numbers;
	    \STATE $q_k\leftarrow\frac{t_i+t_j}{2}$, $t_i\leftarrow q_k$, $r\leftarrow\{1,2,3\}\setminus\{i,j\}$;
	    \STATE $s_j\leftarrow\frac{s_j-s_i}{2}, s_r\leftarrow\frac{s_r}{2}$;
	    \ENDFOR
		\RETURN $\{q_k\}_{k=1}^l$;
	\end{algorithmic}
\end{algorithm}

\begin{theorem}\label{sec4:thm1}
Algorithm \ref{sec4:alg1} is correct.
\end{theorem}
\begin{proof}
Denote the initial values of $s_1,s_2,s_3$, $t_1,t_2,t_3$ respectively by $s_1^{0},s_2^{0},s_3^{0}$, $t_1^{0},t_2^{0},t_3^{0}$, and the values after the $k$-th iteration of the loop respectively by $s_1^{k},s_2^{k},s_3^{k}$, $t_1^{k},t_2^{k},t_3^{k}$. We prove that
\begin{equation}\label{sec4:eq1}
	2^{l-k}=s_1^k+s_2^k+s_3^k\text{ and }2^{l-k}q=s_1^kt_1^k+s_2^kt_2^k+s_3^kt_3^k
\end{equation}
hold true for $k=0,1,\ldots,l$ by induction on $k$. By initialization, we clearly have $2^{l}=q+p-q+2^l-p=s_1^0+s_2^0+s_3^0$ and $2^lq=q\cdot p+(p-q)\cdot0+(2^l-p)\cdot q=s_1^0t_1^0+s_2^0t_2^0+s_3^0t_3^0$. Assume 
that $2^{l-k}=s_1^k+s_2^k+s_3^k$ and $2^{l-k}q=s_1^kt_1^k+s_2^kt_2^k+s_3^kt_3^k$ are true for some $k\ge0$. Now consider the case of $k+1$. We note first that at Step 5 of Algorithm \ref{sec4:alg1}, such $i,j$ indeed exist by construction.
Then, $s_1^{k+1}+s_2^{k+1}+s_3^{k+1}=s_i^k+\frac{s_j^k-s_i^k}{2}+\frac{s^k_r}{2}=\frac{1}{2}(s_1^k+s_2^k+s_3^k)=2^{l-k-1}$ and $s_1^{k+1}t_1^{k+1}+s_2^{k+1}t_2^{k+1}+s_3^{k+1}t_3^{k+1}=s_i^k\cdot\frac{t_i^k+t_j^k}{2}+\frac{s_j^k-s_i^k}{2}\cdot t_j^k+\frac{s^k_r}{2}\cdot t_r^k=\frac{1}{2}(s_1^kt_1^k+s_2^kt_2^k+s_3^kt_3^k)=2^{l-k-1}q$. So we complete the induction.

Letting $k=l$ in \eqref{sec4:eq1}, we obtain $1=s_1^l+s_2^l+s_3^l$ and $q=s_1^lt_1^l+s_2^lt_2^l+s_3^lt_3^l$. Because $s_1^l,s_2^l,s_3^l$ are nonnegative integers, the equality $1=s_1^l+s_2^l+s_3^l$ implies $s_i^l=1,s_j^l=s_r^l=0$, and hence $q=t_i^l=q_l$. From this and by construction, we see that $\{q_k\}_{k=1}^l$ is indeed a $(p,q)$-mediated sequence. Moreover, by Theorem \ref{sec3:thm2}, $\{q_k\}_{k=1}^l$ is minimum. Thus we have proved the correctness of Algorithm \ref{sec4:alg1}.
\end{proof}

\begin{remark}
Algorithm \ref{sec4:alg1} can be readily adapted to produce an optimal simple second-order cone representation for the bivariate weighted geometric mean inequality $x_1^{q}x_2^{p-q}\ge x_{3}^{p}$ with $0<q<p$.
\end{remark}

\begin{remark}
The essence of Algorithm \ref{sec4:alg1} has appeared in the proof of Proposition 5 of \cite{morenko2013p} in the context of second-order cone representations for trivariate weighted geometric mean inequalities $x_1^{s_1}x_2^{s_2}x_3^{s_3}\ge x_{4}^{2^l}$ with $s_1+s_2+s_3=2^l$.
\end{remark}

Note that for integers $0<q<p$, $L(q,p-q)$ denotes the cardinality of a minimum $(p,q)$-mediated sequence. By Theorem \ref{sec4:thm1} we immediately obtain the exact value of $L(q,p-q)$, which resolves a conjecture concerning the value of $L(q,p-q)$ proposed in \cite{magron2023sonc}.
\begin{corollary}\label{sec4:thm2}
For integers $0<q<p$ with $(p,q)=1$, it holds 
\begin{equation}
    L(q,p-q)=\left\lceil\log_2\left(p\right)\right\rceil.
\end{equation}
\end{corollary}

Corollary \ref{sec4:thm2} allows us to further provide an upper bound on $L(s_1,\ldots,s_m)$.
\begin{corollary}\label{sec4:thm4}
Let $s_1,\ldots,s_m\in\N^*$ be a tuple of integers with $(s_1,\ldots,s_m)=1$. Then 
\begin{equation}
L(s_1,\ldots,s_m)\le\min_{\sigma\in S_m}\left\{ \sum_{i=1}^{m-1}\left\lceil\log_2\left(\frac{\sum_{j=i}^{m}s_{\sigma(j)}}{\left(\sum_{j=i}^{m}s_{\sigma(j)},s_{\sigma(i)}\right)}\right)\right\rceil\right\},
\end{equation}
where $S_m$ is the symmetry group of $\{1,\ldots,m\}$.
\end{corollary}
\begin{proof}
The conclusion follows by iteratively using
\begin{equation*}
	x_1^{s_1}\cdots x_m^{s_m}\ge x_{m+1}^{\hat{s}} \iff \exists y\ge0\text{ s.t. }x_1^{s_1}y^{\sum_{i=2}^ms_i}\ge x_{m+1}^{\hat{s}},x_2^{s_2}\cdots x_m^{s_m}\ge y^{\sum_{i=2}^ms_i}
\end{equation*}
and applying Corollary \ref{sec4:thm2} to the bivariate weighted geometric mean inequality arising at each iteration.
\end{proof}

\begin{example}
By Corollary \ref{sec3:thm3}, $L(4,3,2)\ge4$. By Corollary \ref{sec4:thm4}, $L(4,3,2)\le4$. It follows $L(4,3,2)=4$.
\end{example}

We notice that Algorithm \ref{sec4:alg1} can be readily adapted to produce an optimal simple second-order cone representation for the inequality $x_1^{s_1}x_2^{s_2}x_3^{s_3}\ge x_{4}^{2^l}$ with $s_1+s_2+s_3=2^l$, and so obtain the following corollary.
\begin{corollary}\label{sec4:cor1}
For integers $s_1,s_2,s_3\in\N^*$ with $(s_1,s_2,s_3)=1$, if $s_1+s_2+s_3=2^l$ for some $l\in\N$, then 
\begin{equation}
	L(s_1,s_2,s_3)=l.
\end{equation}
\end{corollary}

The minimum mediated sequence produced by Algorithm \ref{sec4:alg1} actually has a special structure which leads to the next definition.
\begin{definition}\label{sec4:def1}
For integers $0<q<p$ with $(p,q)=1$, a minimum $(p,q)$-mediated sequence $A=\{q_1,\ldots,q_l\}$ with $l=\left\lceil\log_2\left(p\right)\right\rceil$ is \emph{successive} if it can be sorted in such a way that 
\begin{enumerate}[(1)]
    \item $q_i=\frac{q_{i-1}+t_{i}}{2}$ for $i=1,\ldots,l$,
    \item $q_l=q$,
\end{enumerate}
where $t_i\in\{0,p,q,q_1,\ldots,q_{i-2}\}$ for $i=1,\ldots,l$ and $q_{0}\in\{p,q\}$.
\end{definition}

Since the minimum mediated sequence output by Algorithm \ref{sec4:alg1} is successive by construction, we know that successive minimum $(p,q)$-mediated sequences exist for any integers $0<q<p$ with $(p,q)=1$.

A successive minimum mediated sequence has a distinguished property, that is, it can be represented by a particular ``binary tree". To get a quick flavor of this fact, let us begin with an illustrative example. Let $p=57,q=11$, and $A=\{34,17,37,27,22,11\}$. Noting
\begin{equation}\label{sec5:eq0}
    34=\frac{p+q}{2}, 17=\frac{34}{2}, 37=\frac{17+p}{2}, 27=\frac{37+17}{2}, 22=\frac{27+17}{2}, q=\frac{22}{2},
\end{equation}
we see that $A$ is a successive minimum $(p,q)$-mediated sequence. From \eqref{sec5:eq0}, one can easily get the following iterated fraction representation of $q$:
\begin{equation}\label{sec5:eq5}
     q=\dfrac{\dfrac{\dfrac{\dfrac{p+q}{4}+p}{2}+\dfrac{p+q}{4}}{2}+\dfrac{p+q}{4}}{4}.
\end{equation}
The mediated sequence $A$ can be recovered from \eqref{sec5:eq5}, which is visualized by the binary tree displayed in Figure \ref{sec5:fg1}.

\begin{figure}
	\centering
	{\footnotesize
	\begin{tikzpicture}
		\node at (0,0) {$11$};
		\draw (0,0) circle (0.25);
		\fill[fill=red,fill opacity=0.2] (0,0) circle (0.25);
		\draw (0,-0.25)--(0,-0.75);
		\node at (0,-1) {$22$};
		\draw (0,-1) circle (0.25);
		\fill[fill=red,fill opacity=0.2] (0,-1) circle (0.25);
		\draw (-0.15,-1.2)--(-1,-1.75);
		\draw (0.15,-1.2)--(1,-1.75);
		\node at (-1,-2) {$27$};
		\draw (-1,-2) circle (0.25);
		\fill[fill=red,fill opacity=0.2] (-1,-2) circle (0.25);
		\node at (1,-2) {$17$};
		\draw (1,-2) circle (0.25);
		\draw (1,-2.25)--(1,-2.75);
		\node at (1,-3) {$34$};
		\draw (1,-3) circle (0.25);
		\draw (0.85,-3.2)--(0.45,-3.75);
		\draw (1.15,-3.2)--(1.55,-3.75);
		\node at (0.45,-4) {$p$};
		\draw (0.45,-4) circle (0.25);
		\node at (1.55,-4) {$q$};
		\draw (1.55,-4) circle (0.25);
		\draw (-1.15,-2.2)--(-1.65,-2.75);
		\draw (-0.85,-2.2)--(-0.35,-2.75);
		\node at (-0.35,-3) {$17$};
		\draw (-0.35,-3) circle (0.25);
		\draw (-0.35,-3.25)--(-0.35,-3.75);
		\node at (-0.35,-4) {$34$};
		\draw (-0.35,-4) circle (0.25);
		\draw (-0.5,-4.2)--(-0.9,-4.75);
		\draw (-0.2,-4.2)--(0.2,-4.75);
		\node at (-0.9,-5) {$p$};
		\draw (-0.9,-5) circle (0.25);
		\node at (0.2,-5) {$q$};
		\draw (0.2,-5) circle (0.25);
		\node at (-1.65,-3) {$37$};
		\draw (-1.65,-3) circle (0.25);
		\fill[fill=red,fill opacity=0.2] (-1.65,-3) circle (0.25);
		\draw (-1.8,-3.2)--(-2.2,-3.75);
		\draw (-1.5,-3.2)--(-1.1,-3.75);
		\draw (-1.1,-4) circle (0.25);
		\node at (-1.1,-4) {$p$};
		\draw (-2.2,-4) circle (0.25);
		\node at (-2.2,-4) {$17$};
		\fill[fill=red,fill opacity=0.2] (-2.2,-4) circle (0.25);
		\draw (-2.2,-4.25)--(-2.2,-4.75);
		\draw (-2.2,-5) circle (0.25);
		\node at (-2.2,-5) {$34$};
		\fill[fill=red,fill opacity=0.2] (-2.2,-5) circle (0.25);
		\draw (-2.35,-5.2)--(-2.75,-5.75);
		\draw (-2.05,-5.2)--(-1.65,-5.75);
		\node at (-2.75,-6) {$p$};
		\draw (-2.75,-6) circle (0.25);
		\node at (-1.65,-6) {$q$};
		\draw (-1.65,-6) circle (0.25);
	\end{tikzpicture}}
	\caption{The binary tree representation of a successive minimum $(p,q)$-mediated sequence with $p=57,q=11$.}\label{sec5:fg1}
\end{figure}

Taking inspirations from the above example, we can construct a binary tree representation for any successive minimum $(p,q)$-mediated sequence $A=\{q_i\}_{i=1}^l$, where $A$ is sorted according to Definition \ref{sec4:def1}. For simplicity, we assume from now on that $p,q$ are both odd\footnote{The other cases can be easily converted to this case.}. We shall describe the construction in an iterative manner.
Let $T_1$ be the binary tree consisting of a root node labelled by $q_1$ along with two children labelled by $p$ and $q$, respectively. For $2\le i\le l$, we iteratively define
\begin{equation}
	T_i=\begin{cases}
		\cC(q_i,T_{i-1}), &\text{ if }q_i=\frac{q_{i-1}}{2},\\
		\cC(q_i,T_{i-1}, p), &\text{ if }q_i=\frac{q_{i-1}+p}{2},\\
		\cC(q_i,T_{i-1}, q), &\text{ if }q_i=\frac{q_{i-1}+q}{2},\\
		\cC(q_i,T_{i-1}, T_j), &\text{ if }q_i=\frac{q_{i-1}+q_{j}}{2},
	\end{cases}
\end{equation}
where $\cC(q_i,T_{i-1})$ denotes the binary tree obtained by adding a root node labelled by $q_i$ connected to $T_{i-1}$ (viewed as a left subtree); $\cC(q_i,T_{i-1}, p)$ (resp. $\cC(q_i,T_{i-1}, q)$) denotes the binary tree obtained by adding a root node labelled by $q_i$ connected to the left to $T_{i-1}$ and to the right to a leaf node labelled by $p$ (resp. $q$); $\cC(q_i,T_{i-1}, T_j)$ denotes the binary tree obtained by adding a root node labelled by $q_i$ connected to the left to $T_{i-1}$ and to the right to $T_{j}$. We then say that $T_l$ is the \emph{binary tree representation} of $A$.
For a binary tree representation $T$, we can naturally define the height of any node such that the node at the bottom is of height $0$. The height of $T$ or its subtree is the height of the root node. We use $p(T)$ (resp. $q(T)$) to denote the set of heights of leaf nodes labelled by $p$ (resp. $q$) of $T$.

\begin{theorem}\label{sec4:thm3}
Suppose that $T$ is the binary tree representation of a successive minimum $(p,q)$-mediated sequence $A$. Then the following hold:
\begin{enumerate}[(1)]
\item The height of $T$ is $\left\lceil\log_2(p)\right\rceil$;
\item Any leaf node of $T$ is labelled by either $p$ or $q$;
\item The root node of $T$ is labelled by $q$;
\item The root node of the left subtree of $T$ of height $i$ is labelled by $q_i$;
\item Any right subtree of $T$ either is a leaf node or coincides with some left subtree;
\item The label of any non-leaf node is the average of labels of its children or the half of the label of its child if there is only one child;
\item $\sum_{j\in p(T)}2^{j}=q$ and $\sum_{j\in q(T)}2^j=2^{\left\lceil\log_2(p)\right\rceil}-p$.
\end{enumerate}
\end{theorem}
\begin{proof}
(1)--(6) are immediate from the construction.

Let $l=\left\lceil\log_2\left(p\right)\right\rceil$ and assume that $A=\{q_i\}_{i=1}^l$ is sorted according to Definition \ref{sec4:def1}. For $i=1,\ldots,l$, let $T_i$ be the left subtree of $T$ with the root node labelled by $q_i$. We claim that 
\begin{equation}\label{sec5:eq7}
	2^iq_i=\left(\sum_{j\in p(T_i)}2^{j}\right)p+\left(\sum_{j\in q(T_i)}2^j\right)q
\end{equation}
holds true for all $i$. Let us prove \eqref{sec5:eq7} by induction on $i$.
For $i=1$, we have $2q_1=p+q$ which is exactly \eqref{sec5:eq7}. Now assume that \eqref{sec5:eq7} is true for $i=k\ge1$. Then if $q_{k+1}=\frac{q_{k}}{2}$, we have
\begin{align*}
	2^{k+1}q_{k+1}=2^{k}q_k&=\left(\sum_{j\in p(T_k)}2^{j}\right)p+\left(\sum_{j\in q(T_k)}2^j\right)q\\
	&=\left(\sum_{j\in p(T_{k+1})}2^{j}\right)p+\left(\sum_{j\in q(T_{k+1})}2^j\right)q;
\end{align*}
if $q_{k+1}=\frac{q_{k}+p}{2}$, we have
\begin{align*}
2^{k+1}q_{k+1}=2^{k}q_k+2^kp&=\left(\sum_{j\in p(T_k)}2^{j}+2^k\right)p+\left(\sum_{j\in q(T_k)}2^j\right)q\\
&=\left(\sum_{j\in p(T_{k+1})}2^{j}\right)p+\left(\sum_{j\in q(T_{k+1})}2^j\right)q;
\end{align*}
if $q_{k+1}=\frac{q_{k}+q}{2}$, we have
\begin{align*}
2^{k+1}q_{k+1}=2^{k}q_k+2^kq&=\left(\sum_{j\in p(T_k)}2^{j}\right)p+\left(\sum_{j\in q(T_k)}2^j+2^k\right)q\\
&=\left(\sum_{j\in p(T_{k+1})}2^{j}\right)p+\left(\sum_{j\in q(T_{k+1})}2^j\right)q;
\end{align*}
if $q_{k+1}=\frac{q_{k}+q_t}{2}$, we have
\begin{align*}
2^{k+1}q_{k+1}&=2^{k}q_k+2^kq_t\\
&=\left(\sum_{j\in p(T_k)}2^{j}+2^{k-t}\sum_{j\in p(T_t)}2^{j}\right)p+\left(\sum_{j\in q(T_k)}2^j+2^{k-t}\sum_{j\in q(T_t)}2^{j}\right)q\\
&=\left(\sum_{j\in p(T_{k+1})}2^{j}\right)p+\left(\sum_{j\in q(T_{k+1})}2^j\right)q.	
\end{align*}
Therefore, \eqref{sec5:eq7} is also true for $i=k+1$, and we complete the induction. Letting $i=l$ in \eqref{sec5:eq7} gives $2^{\left\lceil\log_2(p)\right\rceil}q=\left(\sum_{j\in p(T)}2^{j}\right)p+\left(\sum_{j\in q(T)}2^j\right)q$, from which we deduce $\sum_{j\in p(T)}2^{j}=q$ and $\sum_{j\in q(T)}2^j=2^{\left\lceil\log_2(p)\right\rceil}-p$.
\end{proof}

By virtue of Theorem \ref{sec4:thm3}, we are able to enumerate all successive minimum $(p,q)$-mediated sequences for given $p,q$ via traversing the related binary tree representations. We emphasize that Property (7) stated in Theorem \ref{sec4:thm3} is crucial to reduce the search space of valid binary tree representations. 

\section{Algorithms}\label{sec4}
In this section, we study algorithms for computing simple second-order cone representations of weighted geometric mean inequalities. Unlike the bivariate case treated in the previous section, computing optimal simple second-order cone representations for general weighted geometric mean inequalities seems a notoriously difficult problem. Therefore, we are mostly interested in efficient heuristic algorithms that can produce approximately optimal simple second-order cone representations. We will propose two types of fast heuristic algorithms in Section~\ref{sec5:fh}, and then extend them to certain ``traversal-style'' algorithms in Section~\ref{sec5:ta}. For completeness and comparison, a brute-force algorithm for computing optimal simple second-order cone representations is provided in Section~\ref{sec5:bf}. Finally, we evaluate all these algorithms via numerical experiments in Section~\ref{sec5:ne}.

\subsection{Fast heuristic algorithms}\label{sec5:fh}
Our heuristic algorithms rely on a simple result which is stated in the following lemma.
\begin{lemma}[\cite{kian2019minimal}, Lemma 1]\label{sec5:lm1}
Let $s_1,\ldots,s_m\in\N^*$ be a tuple of integers. Then for a pair $i,j\in\{1,\ldots,m\}$ and for any $\gamma\in\N$ with $0<\gamma\le\min\,\{s_i,s_j\}$, one has
\begin{equation*}
\begin{split}
   \prod_{k=1}^mx_k^{s_k}\ge x_{m+1}^{\hat{s}}\iff\exists y\ge0\text{ s.t. } x_i^{s_i-\gamma}x_j^{s_j-\gamma}y^{2\gamma}\prod_{\substack{k=1\\ k\ne i,j}}^mx_k^{s_k}\ge x_{m+1}^{\hat{s}},x_ix_j\ge y^2.
\end{split}
\end{equation*}
\end{lemma}

Suppose that we are given the inequality $\prod_{k=1}^mx_k^{s_k}\ge x_{m+1}^{\hat{s}}$ and let $l=\left\lceil\log_2\left(\hat{s}\right)\right\rceil$. Let us multiply $\prod_{k=1}^mx_k^{s_k}\ge x_{m+1}^{\hat{s}}$ by $x_{m+1}^{2^l-\hat{s}}$ to obtain $\prod_{k=1}^mx_k^{s_k}x_{m+1}^{2^l-\hat{s}}\ge x_{m+1}^{2^l}$ so that the exponent of $x_{m+1}$ is a power of $2$.
Now we claim that if the exponent of $x_{m+1}$ is a power of $2$, then we are able to obtain a simple second-order cone representation for the weighted geometric mean inequality by iteratively applying Lemma \ref{sec5:lm1} with appropriate $i,j$ and $\gamma$ as detailed in Algorithm~\ref{sec5:alg1}. Note that Algorithm~\ref{sec5:alg1} terminates when the tuple $s_1,\ldots,s_m$ reduces to two nonzero numbers, each being $2^{l-1}$. 

\begin{algorithm}
\renewcommand{\algorithmicrequire}{\textbf{Input:}}
\renewcommand{\algorithmicensure}{\textbf{Output:}}
\caption{}\label{sec5:alg1}
\begin{algorithmic}[1]
\REQUIRE
A tuple of positive integers $s_1,\ldots,s_m$ with $(s_1,\ldots,s_m)=1$
\ENSURE
A configuration $\{(i_k,j_k,t_k)\}_k$ that determines a simple second-order cone representation for $x_1^{s_1}\cdots x_m^{s_m}\ge x_{m+1}^{\hat{s}}$
\STATE $l\leftarrow\left\lceil\log_2\left(\hat{s}\right)\right\rceil$, $t\leftarrow m+1$, $m\leftarrow m+1$, $s_{m}\leftarrow2^l-\hat{s}$;
\STATE $\mathcal{S}\leftarrow\emptyset$;
\REPEAT
\STATE Select a pair $i,j\in\{1,\ldots,m\}$ and an integer $\gamma$ satisfying  $0<\gamma\le\min\,\{s_i,s_j\}$;
\STATE $m\leftarrow m+1$, $s_i\leftarrow s_i-\gamma$, $s_j\leftarrow s_j-\gamma$, $s_m\leftarrow2\gamma$; 
\STATE $\mathcal{S}\leftarrow\mathcal{S}\cup\{(i,j,m)\}$;
\UNTIL{$\gamma=2^{l-1}$}
\RETURN $\mathcal{S}\cup\{(i,j,t)\}$;
\end{algorithmic}
\end{algorithm}

In the following, we propose two strategies to guide us in selecting such a pair $i,j$ and the factor $\gamma$ at Step 4 of Algorithm~\ref{sec5:alg1}. For an integer $r$, let $\Omega(r)$ be the set of exponents of $2$ involved in the binary representation of $r$, and let $\Delta(r)\coloneqq\min\,\Omega(r)$ be the minimal exponent of $2$ involved in the binary representation of $r$. For example, with $r=7$, one has $\Omega(r)=\{0,1,2\}$ and $\Delta(r)=0$.

\begin{theorem}[cf. \cite{kian2019minimal}, Proposition 1]\label{sec5:thm1}
If we select the pair $i,j$ satisfying $\Omega(s_i)\cap\Omega(s_j)\neq\emptyset$ and let $\gamma=\sum_{k\in\Omega(s_i)\cap\Omega(s_j)}2^k$ at Step 4,
then Algorithm~\ref{sec5:alg1} terminates.
\end{theorem}
\begin{proof}
As $s_1+\cdots+s_m=2^l$ holds true at Step 4, we can find a pair $i,j$ such that $\Delta(s_i)=\Delta(s_j)$ and so $\Omega(s_i)\cap\Omega(s_j)\neq\emptyset$.
Let us consider the quantity $h\coloneqq\sum_{i=1}^m|\Omega(s_i)|$. After one iteration, $s_i,s_j$ are replaced by $s_i-\gamma,s_j-\gamma$, and $2\gamma$ is added to the tuple. We then see that $h$ decreases by $|\Omega(s_i)\cap\Omega(s_j)|\ge1$ at each iteration. It follows that the condition $\gamma=2^{l-1}$ is satisfied in at most $\sum_{i=1}^{m}|\Omega(s_i)|-2$ iterations.
Therefore, Algorithm~\ref{sec5:alg1} terminates.
\end{proof}

\begin{corollary}
Let $s_1,\ldots,s_m\in\N^*$ be a tuple of integers with $(s_1,\ldots,s_m)=1$. Then
\begin{equation}
	L(s_1,\ldots,s_m)\le\sum_{i=1}^{m}|\Omega(s_i)|-1.
\end{equation}
\end{corollary}
\begin{proof}
The proof of Theorem \ref{sec5:thm1} immediately implies that Algorithm~\ref{sec5:alg1} yields a simple second-order cone representation for $x_1^{s_1}\cdots x_m^{s_m}\ge x_{m+1}^{\hat{s}}$ of size at most $\sum_{i=1}^{m}|\Omega(s_i)|-1$, from which we get the desired inequality.
\end{proof}

\begin{theorem}\label{sec5:thm2}
If we select the pair $i,j$ such that $\Delta(s_i)=\Delta(s_j)=\min\,\{\Delta(s_k)\}_{k=1}^m$, and let $\gamma=\min\,\{s_i,s_j\}$ at Step 4, then Algorithm~\ref{sec5:alg1} terminates.
\end{theorem}
\begin{proof}
As $s_1+\cdots+s_m=2^l$ holds true at Step 4, we can find a pair $i,j$ such that $\Delta(s_i)=\Delta(s_j)=\min\,\{\Delta(s_k)\}_{k=1}^m$ and so $\Delta(s_i-s_j)-\Delta(s_j)\ge1$.
Let us consider the quantity $h\coloneqq\sum_{i=1}^m(l-\Delta(s_i))$ (set $\Delta(0)\coloneqq l$). After one iteration, $s_i,s_j$ are replaced by $\max\,\{s_i,s_j\}-\min\,\{s_i,s_j\},0$, and $2\min\,\{s_i,s_j\}$ is added to the tuple. We then see that $h$ decreases by $1+\Delta(s_i-s_j)-\Delta(s_j)\ge2$ at each iteration. It follows that the condition $\gamma=2^{l-1}$ is satisfied in at most $\frac{1}{2}\sum_{i=1}^m(l-\Delta(s_i))-1$ iterations.
Therefore, Algorithm~\ref{sec5:alg1} terminates.
\end{proof}

\begin{corollary}
Let $s_1,\ldots,s_m\in\N^*$ be a tuple of integers with $(s_1,\ldots,s_m)=1$ and let $l=\left\lceil\log_2\left(\hat{s}\right)\right\rceil$. Then
\begin{equation}
	L(s_1,\ldots,s_m)\le\frac{1}{2}\sum_{i=1}^m\left(l-\Delta(s_i)\right).
\end{equation}
\end{corollary}
\begin{proof}
The proof of Theorem \ref{sec5:thm2} immediately implies that Algorithm~\ref{sec5:alg1} yields a simple second-order cone representation for $x_1^{s_1}\cdots x_m^{s_m}\ge x_{m+1}^{\hat{s}}$ of size at most $\frac{1}{2}\sum_{i=1}^m(l-\Delta(s_i))$. Thus we get the desired inequality.
\end{proof}

Theorems \ref{sec5:thm1} and \ref{sec5:thm2} allow us to find an approximately optimal simple second-order cone representation for a weighted geometric mean inequality by implementing greedy strategies at each iteration of the loop in Algorithm~\ref{sec5:alg1}: 
\begin{enumerate}[(1)]
	\item selecting the pair $i,j$ to maximize $|\Omega(s_i)\cap\Omega(s_j)|$ and letting $\gamma=\sum_{k\in\Omega(s_i)\cap\Omega(s_j)}2^k$ (hereafter referred to as the greedy-common-one strategy), or alternatively,
	\item selecting the pair $i,j$ to maximize $\Delta(s_i-s_j)$ with $\Delta(s_i)=\Delta(s_j)=\min\,\{s_k\}_{k=1}^m$ and letting $\gamma=\min\,\{s_i,s_j\}$ (hereafter referred to as the greedy-power-two strategy). 
\end{enumerate}

\begin{remark}
Algorithm~\ref{sec5:alg1} that employs the greedy-common-one strategy has been used to compute simple second-order cone representations
for weighted geometric mean inequalities $x_1^{s_1}\cdots x_m^{s_m}\ge x_{m+1}^{\hat{s}}$ with $\hat{s}=2^l$ in \cite{kian2019minimal}.
\end{remark}

The following lemmas whose proofs are straightforward allow us to further enhance the performance of Algorithm~\ref{sec5:alg1}.
\begin{lemma}\label{sec5:lm2}
Let $s_1,\ldots,s_m\in\N^*$ be a tuple of integers. Assume $s_m=\max\,\{s_i\}_{i=1}^m$ and $l=\left\lceil\log_2\left(\hat{s}\right)\right\rceil$. Then,
\begin{align*}
&\prod_{k=1}^mx_k^{s_k}\ge x_{m+1}^{\hat{s}}\iff\exists y\ge0\text{ s.t. }\\
&\begin{cases}
   \prod_{k=1}^{m-1}x_k^{s_k}\ge y^{\frac{\hat{s}}{2}}, x_my\ge x_{m+1}^2, &\text{ if }s_m=\frac{\hat{s}}{2};\\
   \prod_{k=1}^{m-1}x_k^{s_k}x_{m+1}^{2s_m-\hat{s}}\ge y^{s_m}, x_my\ge x_{m+1}^2, &\text{ if }\frac{\hat{s}}{2}<s_m\le 2^{l-1};\\
   \prod_{k=1}^{m-1}x_k^{s_k}x_m^{s_m-2^{l-1}}x_{m+1}^{2^l-\hat{s}}\ge y^{2^{l-1}}, x_my\ge x_{m+1}^2, &\text{ if }s_m>2^{l-1},\hat{s}<2^l;\\
   \prod_{k=1}^{m-1}x_k^{s_k}x_m^{s_m-2^{l-1}}\ge y^{2^{l-1}}, x_my\ge x_{m+1}^2, &\text{ if }s_m>2^{l-1},\hat{s}=2^l.
\end{cases}
\end{align*}
\end{lemma}

\begin{lemma}\label{sec5:lm3}
Let $s_1,\ldots,s_m\in\N^*$ be a tuple of integers and assume $s_1=s_2$. Then,
\begin{equation*}
\prod_{k=1}^mx_k^{s_k}\ge x_{m+1}^{\hat{s}}\iff\exists y\ge0\text{ s.t. }
   \prod_{k=3}^{m}x_k^{s_k}y^{2s_2}\ge x_{m+1}^{\hat{s}}, x_1x_2\ge y^2.
\end{equation*}
\end{lemma}

\begin{lemma}\label{sec5:lm4}
Let $s_1,\ldots,s_m\in\N^*$ be a tuple of integers. Assume $l=\left\lceil\log_2\left(\hat{s}\right)\right\rceil$ and $s_1$ is the unique odd number among $s_1,\ldots,s_m$, which satisfies $s_1\le 2^l-\hat{s}$. Then,
\begin{equation*}
\prod_{k=1}^mx_k^{s_k}\ge x_{m+1}^{\hat{s}}\iff\exists y\ge0\text{ s.t. }
   \prod_{k=2}^{m}x_k^{\frac{s_k}{2}}y^{s_1}\ge x_{m+1}^{\frac{\hat{s}+s_1}{2}}, x_1x_{m+1}\ge y^2.
\end{equation*}
\end{lemma}

By invoking any of Lemmas \ref{sec5:lm2}--\ref{sec5:lm4} if appliable, we get either a reduction of $\hat{s}$ by a factor $2$ or a decrease of $m$ by $1$ at the price of one quadratic inequality. We thereby embed Lemmas \ref{sec5:lm2}--\ref{sec5:lm4} into Algorithm~\ref{sec5:alg1} as detailed in Algorithm \ref{sec5:alg2}.

\begin{algorithm}
\renewcommand{\algorithmicrequire}{\textbf{Input:}}
\renewcommand{\algorithmicensure}{\textbf{Output:}}
\caption{{\tt Heuristic}$(s_1,\ldots,s_m; t=m+1)$}\label{sec5:alg2}
\begin{algorithmic}[1]
\REQUIRE
A tuple of nonnegative integers $s_1,\ldots,s_m$ and a positive integer $t$ equaling $m+1$ by default
\ENSURE
A configuration $\{(i_k,j_k,t_k)\}_k$ that determines a simple second-order cone representation for $x_1^{s_1}\cdots x_m^{s_m}\ge x_{t}^{\hat{s}}$
\IF{$t=m+1$}
\STATE $m\leftarrow m+1$, $s_m\leftarrow0$;
\ENDIF
\STATE $l\leftarrow\left\lceil\log_2\left(\hat{s}\right)\right\rceil$;
\IF{$\exists i,j\in\{1,\ldots,m\}$ such that $s_i=s_j>0$}
\IF{$s_i\ne2^{l-1}$}
\STATE $m\leftarrow m+1$, $s_{m}\leftarrow 2s_i$, $s_i\leftarrow0$, $s_j\leftarrow0$;
\RETURN {\tt Heuristic}$(s_1,\ldots,s_m; t)\cup\{(i,j,m)\}$;
\ELSE
\RETURN $\{(i,j,t)\}$;
\ENDIF
\ENDIF
\STATE Find $k\in\{1,\ldots,m\}$ such that $s_k=\max\,\{s_i\}_{i=1}^m$;
\IF{$2s_k\ge\hat{s}$}
\IF{$2s_k=\hat{s}$}
\STATE $s_k\leftarrow0$;
\ELSIF{$s_k\le2^{l-1}$}
\STATE $s_{t}\leftarrow2s_k-\hat{s},s_k\leftarrow0$;
\ELSIF{$\hat{s}<2^{l}$}
\STATE $s_{t}\leftarrow2^l-\hat{s},s_{k}\leftarrow s_k-2^{l-1}$;
\ELSE
\STATE $s_{k}\leftarrow s_k-2^{l-1}$;
\ENDIF
\RETURN {\tt Heuristic}$(s_1,\ldots,s_m; m+1)\cup\{(k,m+1,t)\}$;
\ENDIF
\IF{$\argmin\,\{\Delta(s_i)\}_{i=1}^m=\{r\}$ and $s_r\le2^l-\hat{s}$}
\STATE $m\leftarrow m+1,s_{m}\leftarrow2s_r,s_r\leftarrow0$;
\RETURN {\tt Heuristic}$(s_1,\ldots,s_m; t)\cup\{(r,t,m)\}$;
\ENDIF
\STATE $s_{t}\leftarrow2^l-\hat{s}$;
\IF{$\exists i\in\{1,\ldots,m\}, i\ne t$ such that $s_t=s_i>0$}
\STATE $m\leftarrow m+1$, $s_{m}\leftarrow2s_i$, $s_i\leftarrow0$, $s_t\leftarrow0$;
\RETURN {\tt Heuristic}$(s_1,\ldots,s_m; t)\cup\{(i,t,m)\}$;
\ENDIF
\STATE Select a pair $i,j\in\{1,\ldots,m\}$ and an integer $\gamma$ satisfying  $0<\gamma\le\min\,\{s_i,s_j\}$;
\STATE $m\leftarrow m+1$, $s_{m}\leftarrow 2\gamma$, $s_i\leftarrow s_i-\gamma$, $s_j\leftarrow s_j-\gamma$;
\RETURN {\tt Heuristic}$(s_1,\ldots,s_m; t)\cup\{(i,j,m)\}$;
\end{algorithmic}
\end{algorithm}


\subsection{Traversal algorithms}\label{sec5:ta}
Instead of considering only the ``maximal'' pair at each iteration (Step 4 of Algorithm~\ref{sec5:alg1}), we may take into account all pairs $i,j$ such that $\Omega(s_i)\cap\Omega(s_j)\neq\emptyset$ (hereafter referred to as the common-one strategy) or $\Delta(s_i)=\Delta(s_j)=\min\,\{s_k\}_{k=1}^m$ (hereafter referred to as the power-two strategy). Such traversal algorithms enable us to obtain a simple second-order cone representation of possibly smaller size for a weighted geometric mean inequality by spending more time.

\subsection{A brute-force algorithm}\label{sec5:bf}
Besides the heuristics aiming to efficiently produce approximately optimal simple second-order cone representations, we also propose a brute-force algorithm to compute an exact optimal simple second-order cone representation for a weighted geometric mean inequality.

Let the configuration $\{(i_k,j_k,m+k)\}_{k=1}^n$ determine a simple second-order cone representation for $x_1^{s_1}\cdots x_m^{s_m}\ge x_{m+1}^{\hat{s}}$ (equivalently, an $(\A,\a_{m+1})$-mediated set with $\A=\{\a_i\}_{i=1}^m$ being a trellis and $\a_{m+1}=\sum_{i=1}^{m}\frac{s_i}{\hat{s}}\a_i$) so that \eqref{sec3:eq2} holds.
Denote the coefficient matrix of \eqref{sec3:eq2} by $C$ and let $A=[\a_1,\ldots,\a_{m+n}]^{\intercal}$ with $\a_i$'s being viewed as column vectors, so that we can rewrite \eqref{sec3:eq2} in matrix form as $CA=0$.
By construction, the equality $\sum_{i=1}^{m}s_i\a_i=\hat{s}\a_{m+1}$ is implied by \eqref{sec3:eq2}, which means that the vector $(s_1,\ldots,s_m,-\hat{s},0,\ldots,0)$ belongs to the row space of $C$, i.e., there exist $\gamma_1,\ldots,\gamma_n\in\Q$ such that $(s_1,\ldots,s_m,-\hat{s},0,\ldots,0)=\sum_{t=1}^n\gamma_t\mathbf{c}_t$ with $\{\mathbf{c}_t\}_{t=1}^n$ being the row vectors of $C$.
It follows that the following linear system in variables $\gamma_1,\ldots,\gamma_n$ admits a rational solution:
\begin{equation}\label{sec5-eq1}
	\begin{cases}
		\sum_{i_t=k}\gamma_t+\sum_{j_t=k}\gamma_t=s_k,\quad k=1,\ldots,m,\\
		\sum_{i_t=k}\gamma_t+\sum_{j_t=k}\gamma_t=2\gamma_1-\hat{s},\quad k=m+1,\\
		\sum_{i_t=k}\gamma_t+\sum_{j_t=k}\gamma_t=2\gamma_{k-m},\quad k=m+2,\ldots,m+n.\\
	\end{cases}
\end{equation}
Conversely, if the linear system \eqref{sec5-eq1} admits a rational solution for some configuration $\{(i_k,j_k,m+k)\}_{k=1}^n$, then this configuration gives rise to \eqref{sec3:eq2}
and hence determines an $(\A,\a_{m+1})$-mediated set with $\a_{m+1}=\sum_{i=1}^{m}\frac{s_i}{\hat{s}}\a_i$ as well as a simple second-order cone representation for $x_1^{s_1}\cdots x_m^{s_m}\ge x_{m+1}^{\hat{s}}$. Building upon these facts, we give the brute-force algorithm for computing optimal simple second-order cone representations in Algorithm~\ref{sec5:alg3}.

\begin{algorithm}
	\renewcommand{\algorithmicrequire}{\textbf{Input:}}
	\renewcommand{\algorithmicensure}{\textbf{Output:}}
	\caption{{\tt Bruteforce}$(s_1,\ldots,s_m)$}\label{sec5:alg3}
	\begin{algorithmic}[1]
		\REQUIRE
		A tuple of positive integers $s_1,\ldots,s_m$ with $(s_1,\ldots,s_m)=1$
		\ENSURE A configuration $\{(i_k,j_k,m+k)\}_{k=1}^n$ that determines an optimal simple second-order cone representation for $x_1^{s_1}\cdots x_m^{s_m}\ge x_{m+1}^{\hat{s}}$
		\STATE $n\leftarrow\max\,\{\left\lceil\log_2\left(\hat{s}\right)\right\rceil,m-1\}$;  \COMMENT{Remark \ref{sec3:rm1}}
		\STATE Enumerate all legitimate configurations $\{(i_k,j_k,m+k)\}_{k=1}^n$ which are denoted by $\mathcal{T}_{m,n}$;
		\FOR{each $\mathcal{S}$ in $\mathcal{T}_{m,n}$}
		\IF{The system \eqref{sec5-eq1} is feasible}
		\RETURN $\mathcal{S}$;
		\ENDIF
		\ENDFOR
		\STATE $n\leftarrow n+1$;
		\STATE \textbf{goto} Step 2;
	\end{algorithmic}
\end{algorithm}

The most expensive part of Algorithm~\ref{sec5:alg3} is Step 2 as for given $m,n$, the number of legitimate configurations $\{(i_k,j_k,m+k)\}_{k=1}^n$ might be very large. In order to speed up the enumeration, we thereby impose some conditions on legitimate configurations, which are stated in the following proposition.

\begin{proposition}\label{sec5:prop1}
There is no loss of generality in assuming that any legitimate configuration $\{(i_k,j_k,m+k)\}_{k=1}^n$ satisfies the following conditions:
\begin{enumerate}[(1)]
	\item $\{1,\ldots,m\}\cup\{m+2,\ldots,m+n\}\subseteq\cup_{k=1}^n\{i_k,j_k\}$;
	\item $i_k<j_k$;
	\item $i_k,j_k\ne m+k$;
	\item $(i_{k_1},j_{k_1})\ne(i_{k_2},j_{k_2})$ if $k_1\ne k_2$;
	\item $\{i_{k_1},j_{k_1},m+k_1\}\ne\{i_{k_2},j_{k_2},m+k_2\}$ if $k_1\ne k_2$;
	\item $t_k\ge m+k+1$ for $k=1,\ldots,n-1$;
	\item $i_1\le m,j_1=m+2$ or $i_1=m+2,j_1=m+3$;
	\item $i_k\le t_{k-1}, j_k\le t_{k-1}+1$ or $i_k=t_{k-1}+1, j_k=t_{k-1}+2$ for $k=2,\ldots,n$,
\end{enumerate}
where $t_k\coloneqq\max\,\{j_1,\ldots,j_k\}$ for $k=1,\ldots,n-1$.
\end{proposition}
\begin{proof}
(1) follows from the proof of Theorem \ref{sec3:thm1}. (2)--(5) are immediate from the definition. (6) is due to the fact that we are seeking a configuration of minimum size, and if $t_k< m+k+1$ for some $k\in\{1,\ldots,n-1\}$, then $\{(i_l,j_l,m+l)\}_{l=1}^{k}$ is a configuration of smaller size. (7)--(8) are because we can arbitrarily label the auxiliary variables $x_{m+2},\ldots,x_{m+n}$.
\end{proof}

Table \ref{tb3} shows the cardinality of $\mathcal{T}_{m,n}$ derived from Proposition \ref{sec5:prop1} for different $(m,n)$. We could see that $|\mathcal{T}_{m,n}|$ grows rapidly with $m,n$. However, for each $(m,n)$, the set $\mathcal{T}_{m,n}$ needs to be computed just once and then can be used forever.

\begin{table}[htbp]
	\caption{The number of configurations that satisfy the conditions in Proposition \ref{sec5:prop1}.}\label{tb3}
	\renewcommand\arraystretch{1.2}
	\centering
	\resizebox{\linewidth}{!}{
	\begin{tabular}{c|c|c|c|c|c|c|c|c|c}
		$(m,n)$&$(3,2)$&$(3,3)$&$(3,4)$&$(3,5)$&$(3,6)$&$(4,3)$&$(4,4)$&$(4,5)$&$(4,6)$\\
		\hline
		$|\mathcal{T}_{m,n}|$&$3$&$48$&$828$&$17178$&$419559$&$18$&$588$&$17016$&$514524$\\
	\end{tabular}}
\end{table}

\subsection{Numerical experiments}\label{sec5:ne}
The algorithms discussed above were implemented in the Julia package {\tt MiniSOC}, which is available at \href{https://github.com/wangjie212/MiniSOC}{https://github.com/wangjie212/MiniSOC}. Numerical experiments were performed on a desktop computer with Windows 10 system, Intel(R) Core(TM) i9-10900 CPU@2.80GHz and 32G RAM.

\subsubsection{Evaluating different heuristic algorithms}

To test the perfomance of different heuristic algorithms, we run them on the partitions $s_1+\cdots+s_m$ of the integer $\hat{s}=83$. The partitions are restricted to be of length $m=3,4,5,6$ and satisfy $(s_1,\ldots,s_m)=1$. We denote Algorithm \ref{sec5:alg2} implementing the greedy-common-one strategy (resp. the greedy-power-two strategy) by {\tt GreedyCommone} (resp. {\tt GreedyPowertwo}); we denote the traversal algorithm implementing the common-one strategy (resp. the power-two strategy) by {\tt TraversalCommone} (resp. {\tt TraversalPowertwo}). The results are reported in Table \ref{tb1}. The data in Table \ref{tb1} show that: (1) the greedy algorithms are significantly faster than the traversal algorithms when $m\ge5$; (2) the traversal algorithms may produce simple second-order cone representations of smaller size than the greedy algorithms; (3) {\tt GreedyPowertwo} not only runs faster than {\tt GreedyCommone} but also produces simple second-order cone representations of smaller size, and similar statement also applies to {\tt TraversalPowertwo} and {\tt TraversalCommone}.

\begin{table}[htbp]
\caption{Comparison of heuristic algorithms with $\hat{s}=83$. For each $m$, the first column indicates the sum of sizes of simple second-order cone representations over partitions of length $m$, and the second column indicates the total running time in seconds. The symbol - means running time $>1$ day.}\label{tb1}
\renewcommand\arraystretch{1.2}
\centering
\resizebox{\linewidth}{!}{
\begin{tabular}{c|c|c|c|c|c|c|c|c}
Algorithm&\multicolumn{2}{c|}{$m=3$}&\multicolumn{2}{c|}{$m=4$}&\multicolumn{2}{c|}{$m=5$}&\multicolumn{2}{c}{$m=6$}\\
\hline
{\tt GreedyPowertwo}&$4567$&$0.001$&$37996$&$0.01$&$196262$&$0.11$&$697083$&$0.40$\\
\hline
{\tt GreedyCommone}&$4695$&$0.025$&$39625$&$0.35$&$204927$&$2.36$&$728705$&$10.3$\\
\hline
{\tt TraversalPowertwo}&$4561$&$0.007$&$37648$&$0.25$&$192654$&$9.82$&$681705$&$90.0$\\
\hline
{\tt TraversalCommone}&$4660$&$0.311$&$38894$&$1074$&-&-&-&-\\
\end{tabular}}
\end{table}

We mention that for $m=3,4,5,6$, the number of different partitions of $83$ are $574$, $4109$, $18487$, $58767$, respectively, and the average sizes of second-order cone representations produced by the algorithm {\tt GreedyPowertwo} are $8.0$, $9.2$, $10.6$, $11.9$, respectively.

To further compare the algorithm {\tt GreedyPowertwo} with {\tt GreedyCommone}, we generate random integers $s_1,\ldots,s_m\in(0,10^t)$ with $m=10$. The related results are reported in Table \ref{tb7}, which confirms our observation that {\tt GreedyPowertwo} not only runs faster than {\tt GreedyCommone} (by a factor $\sim100$) but also produces second-order cone representations of smaller size.

\begin{table}[htbp]
\caption{Results of {\tt GreedyPowertwo} and {\tt GreedyCommone} on random instances. Each $t$ corresponds to three different random trials. For each algorithm, the first column indicates the size of second-order cone representation, and the second column indicates the running time in seconds.}\label{tb7}
\renewcommand\arraystretch{1.2}
\centering
\begin{tabular}{c|c|c|c|c}
$t$&\multicolumn{2}{c|}{{\tt GreedyPowertwo}}&\multicolumn{2}{c}{{\tt GreedyCommone}}\\
\hline
10&$91$&$0.0002$&$98$&$0.02$\\
&$90$&$0.0002$&$112$&$0.02$\\
&$91$&$0.0002$&$114$&$0.02$\\
12&$107$&$0.0002$&$135$&$0.03$\\
&$109$&$0.0003$&$119$&$0.03$\\
&$108$&$0.0003$&$121$&$0.04$\\
14&$133$&$0.0004$&$139$&$0.04$\\
&$122$&$0.0003$&$155$&$0.04$\\
&$128$&$0.0003$&$136$&$0.05$\\
16&$146$&$0.0004$&$179$&$0.07$\\
&$140$&$0.0004$&$152$&$0.06$\\
&$146$&$0.0004$&$195$&$0.07$\\
\end{tabular}
\end{table}

\subsubsection{Comparison with optimal simple second-order cone representations}

We provide the sizes of optimal simple second-order cone representations for trivariate weighted geometric mean inequalities $x_1^{s_1}x_2^{s_2}x_3^{s_3}\ge x_{4}^{\hat{s}}$ with $\hat{s}\le15$ in Table \ref{tb2}, which are obtained with the brute-force algorithm (Algorithm~\ref{sec5:alg3}).

\begin{sidewaystable}[htbp]
\caption{Sizes of optimal simple second-order cone representations for trivariate weighted geometric mean inequalities $x_1^{s_1}x_2^{s_2}x_3^{s_3}\ge x_{4}^{\hat{s}}$ (labelled by $(s_1,s_2,s_3)$) with $\hat{s}\le15$. We exclude the cases $\hat{s}=4,8$ in view of Corollary \ref{sec4:cor1}.}\label{tb2}
\renewcommand\arraystretch{1.2}
\centering
\begin{tabular}{l|c|c|c|c|c|c|c|c|c|c}
\multirow{2}{*}{$\hat{s}=3$}&$(1,1,1)$&&&&&&&&&\\
\cline{2-11}
&3&&&&&&&&&\\
\hline
\multirow{2}{*}{$\hat{s}=5$}&$(2,2,1)$&$(3,1,1)$&&&&&&&&\\
\cline{2-11}
&4&4&&&&&&&&\\
\hline
\multirow{2}{*}{$\hat{s}=6$}&$(3,2,1)$&$(4,1,1)$&&&&&&&&\\
\cline{2-11}
&3&3&&&&&&&&\\
\hline
\multirow{2}{*}{$\hat{s}=7$}&$(3,2,2)$&$(3,3,1)$&$(4,2,1)$&$(5,1,1)$&&&&&&\\
\cline{2-11}
&4&4&3&4&&&&&&\\
\hline
\multirow{2}{*}{$\hat{s}=9$}&$(4,3,2)$&$(4,4,1)$&$(5,2,2)$&$(5,3,1)$&$(6,2,1)$&$(7,1,1)$&&&&\\
\cline{2-11}
&4&5&5&5&4&5&&&&\\
\hline
\multirow{2}{*}{$\hat{s}=10$}&$(4,3,3)$&$(5,3,2)$&$(5,4,1)$&$(6,3,1)$&$(7,2,1)$&$(8,1,1)$&&&&\\
\cline{2-11}
&4&4&4&4&4&4&&&&\\
\hline
\multirow{2}{*}{$\hat{s}=11$}&$(4,4,3)$&$(5,3,3)$&$(5,4,2)$&$(5,5,1)$&$(6,3,2)$&$(6,4,1)$&$(7,2,2)$&$(7,3,1)$&$(8,2,1)$&$(9,1,1)$\\
\cline{2-11}
&5&5&4&5&4&4&5&5&4&5\\
\hline
\multirow{2}{*}{$\hat{s}=12$}&$(5,4,3)$&$(5,5,2)$&$(6,5,1)$&$(7,3,2)$&$(7,4,1)$&$(8,3,1)$&$(9,2,1)$&$(10,1,1)$&&\\
\cline{2-11}
&4&4&4&4&4&4&4&4&&\\
\hline
\multirow{4}{*}{$\hat{s}=13$}&$(5,4,4)$&$(5,5,3)$&$(6,4,3)$&$(6,5,2)$&$(6,6,1)$&$(7,3,3)$&$(7,4,2)$&$(7,5,1)$&$(8,3,2)$&$(8,4,1)$\\
\cline{2-11}
&5&5&4&5&5&5&4&5&4&4\\
\cline{2-11}
&$(9,2,2)$&$(9,3,1)$&$(10,2,1)$&$(11,1,1)$&&&&&&\\
\cline{2-11}
&5&5&5&5&&&&&&\\
\hline
\multirow{4}{*}{$\hat{s}=14$}&$(5,5,4)$&$(6,5,3)$&$(7,4,3)$&$(7,5,2)$&$(7,6,1)$&$(8,3,3)$&$(8,5,1)$&$(9,3,2)$&$(9,4,1)$&$(10,3,1)$\\
\cline{2-11}
&4&4&4&4&4&4&4&5&4&5\\
\cline{2-11}
&$(11,2,1)$&$(12,1,1)$&&&&&&&&\\
\cline{2-11}
&4&4&&&&&&&&\\
\hline
\multirow{4}{*}{$\hat{s}=15$}&$(6,5,4)$&$(7,4,4)$&$(7,5,3)$&$(7,6,2)$&$(7,7,1)$&$(8,4,3)$&$(8,5,2)$&$(8,6,1)$&$(9,4,2)$&$(9,5,1)$\\
\cline{2-11}
&5&5&6&5&5&4&4&4&4&5\\
\cline{2-11}
&$(10,3,2)$&$(10,4,1)$&$(11,2,2)$&$(11,3,1)$&$(12,2,1)$&$(13,1,1)$&&&&\\
\cline{2-11}
&5&4&5&5&4&5&&&&\\
\end{tabular}
\end{sidewaystable}

According to Table \ref{tb2}, we emphasize that the four heuristic algorithms produce optimal simple second-order cone representations for all trivariate weighted geometric mean inequalities $x_1^{s_1}x_2^{s_2}x_3^{s_3}\ge x_{4}^{\hat{s}}$ with $\hat{s}\le15$ but $4$ instances corresponding to $(s_1,s_2,s_3)=(5,4,3)$, $(7,3,2)$, $(6,5,3)$, $(11,2,1)$, respectively. For $(s_1,s_2,s_3)=(5,4,3)$, $(7,3,2)$ and $(11,2,1)$, the four heuristic algorithms yield simple second-order cone representations of size $5$, while the optimal size is $4$; for $(s_1,s_2,s_3)=(6,5,3)$, {\tt GreedyCommone} and {\tt TraversalCommone} yield simple second-order cone representations of size $6$, and {\tt GreedyPowertwo} and {\tt TraversalPowertwo} yield simple second-order cone representations of size $5$, while the optimal size is $4$.

\section{Applications}\label{sec5}
In this section, we give three applications of the proposed algorithm {\tt GreedyPowertwo} in polynomial optimization, matrix optimization, quantum information, respectively. All numerical experiments were performed on a desktop computer with Windows 10 system, Intel(R) Core(TM) i9-10900 CPU@2.80GHz and 32G RAM.
\subsection{SONC optimization}
Let $\R[\x]=\R[x_1,\ldots,x_n]$ be the ring of real $n$-variate polynomials. For $\a=(\alpha_i)_i\in\N^n$, let $\x^{\a}\coloneqq x_1^{\alpha_1}\cdots x_n^{\alpha_n}$. Suppose that $\A=\{\a_1,\ldots,\a_{n+1}\}\subseteq(2\N)^{n}$ is a trellis. A polynomial $f=\sum_{i=1}^{n+1}c_{i}\x^{\a_i}-d\x^{\b}\in\R[\x]$ is called a \emph{circuit polynomial} if $c_{i}>0$ for $i=1,\ldots,n+1$ and $\b\in\Conv(\A)^{\circ}\cap\N^{n}$ \cite{iw}. The nonnegativity of a circuit polynomial $f$ on $\R^n$ can be easily verified by
\begin{equation}\label{sec6:eq3}
	f\ge0\iff 
	\begin{cases}
		d<0\text{ or }\prod_{i=1}^{n+1}(c_{i}/\lambda_{i})^{\lambda_{i}}\ge d\ge0,&\text{ if }\b\in(2\N)^n,\\
		\prod_{i=1}^{n+1}(c_{i}/\lambda_{i})^{\lambda_{i}}\ge |d|,&\text{ if }\b\notin(2\N)^n,\\
	\end{cases}
\end{equation}
where $(\lambda_{i})_{i=1}^{n+1}\in\Q^{n+1}_+$ are the barycentric coordinates of $\b$ with respect to $\A$ satisfying $\b=\sum_{i=1}^{n+1}\lambda_{i}\a_i$ and $\sum_{i=1}^{n+1}\lambda_{i}=1$. Note that the conditions in \eqref{sec6:eq3} admit second-order cone representations as the inequalities $\prod_{i=1}^{n+1}(c_{i}/\lambda_{i})^{\lambda_{i}}\ge d\ge0$ are equivalent to
\begin{equation}\label{sec6:eq1}
	\exists y\ge0 \,\text{ s.t. } \prod_{i=1}^{n+1}c_{i}^{\lambda_{i}}\ge y, \, d\prod_{i=1}^{n+1}\lambda_{i}^{\lambda_{i}}=y,
\end{equation}
and the inequality $\prod_{i=1}^{n+1}(c_{i}/\lambda_{i})^{\lambda_{i}}\ge|d|$ is equivalent to
\begin{equation}\label{sec6:eq0}
	\exists y\ge0 \,\text{ s.t. } \prod_{i=1}^{n+1}c_{i}^{\lambda_{i}}\ge y, \,|d|\prod_{i=1}^{n+1}\lambda_{i}^{\lambda_{i}}\le y.
\end{equation}

One can certify the nonnegativity of a polynomial $f$ by decomposing it into a \emph{sum of nonnegative circuit polynomials (SONC)}. Furthermore, one can provide a lower bound on the global minimum of $f$ by solving the following SONC optimization problem:
\begin{equation}\label{sec6:eq2}
	\begin{cases}
		\sup&\gamma\\
		\text{s.t.} &f-\gamma\text{ is a SONC}.
	\end{cases}
\end{equation}
By virtue of \eqref{sec6:eq1}--\eqref{sec6:eq0}, \eqref{sec6:eq2} can be further modeled as a second-order cone program; see \cite{magron2023sonc} for more details.

We select $30$ randomly generated polynomials from the database provided by Seidler and de Wolff in \cite{se}, and solve the related SONC optimization problem \eqref{sec6:eq2} with the SOCP solver {\tt Mosek 9.0}\footnote{The computation was performed with the Julia package {\tt SONCSOCP} which is available at \href{https://github.com/wangjie212/SONCSOCP}{https://github.com/wangjie212/SONCSOCP}.}. For each instance, we use two algorithms to generate the required second-order cone representations: {\tt GreedyPowertwo} and the one proposed in the paper \cite{kian2019minimal}. The results are reported in Table \ref{tb5}. It is evident that the approach with {\tt GreedyPowertwo} is more efficient than the one with the algorithm from \cite{kian2019minimal}, sometimes by orders of magnitude.

\begin{table}[htbp]
\caption{Results for the SONC optimization problem \eqref{sec6:eq2}. $n,d,t$ denote the number of variables, the degree, the number of terms of the polynomial $f$, respectively; the column labelled by ``opt'' indicates the optimum; the columns labelled by ``W'' and ``K-B-G'' indicate the running time in seconds of the approaches with {\tt GreedyPowertwo} and the algorithm from \cite{kian2019minimal}, respectively.}\label{tb5}
\renewcommand\arraystretch{1.2}
\centering
\begin{tabular}{l|c|c|c}
$(n,d,t)$&opt&W&K-B-G\\
\hline
$(10,40,100)$&5.4251&0.35&1.55\\
$(10,40,200)$&7.7161&0.78&3.26\\
$(10,40,300)$&38.662&1.15&4.73\\
\hline
$(10,50,100)$&0.1972&0.35&1.73\\
$(10,50,200)$&4.8642&0.91&4.00\\
$(10,50,300)$&7.0048&1.41&6.09\\
\hline
$(10,60,100)$&2.5232&0.43&2.57\\
$(10,60,200)$&26.017&1.04&4.92\\
$(10,60,300)$&23.417&1.76&8.27\\
\hline
$(20,40,100)$&2.8614&0.47&9.90\\
$(20,40,200)$&9.8565&1.21&24.0\\
$(20,40,300)$&1.9062&2.20&35.6\\
\hline
$(20,50,100)$&2.1435&0.53&12.9\\
$(20,50,200)$&8.1728&1.46&26.8\\
$(20,50,300)$&14.922&2.61&49.2\\
\hline
$(20,60,100)$&4.8874&0.59&16.1\\
$(20,60,200)$&3.0012&1.60&36.7\\
$(20,60,300)$&10.640&2.87&59.7\\
\hline
$(30,40,100)$&1.3083&0.86&45.2\\
$(30,40,200)$&3.3180&2.39&126\\
$(30,40,300)$&9.3044&4.07&198\\
\hline
$(30,50,100)$&3.3983&1.03&80.1\\
$(30,50,200)$&6.0753&2.79&186\\
$(30,50,300)$&1.9555&5.08&279\\
\hline
$(30,60,100)$&0.3946&1.17&83.6\\
$(30,60,200)$&5.8770&2.99&200\\
$(30,60,300)$&6.8931&4.97&348\\
\hline
$(40,50,100)$&3.3078&1.72&266\\
$(40,50,200)$&4.7519&5.50&779\\
$(40,50,300)$&2.9359&9.37&1142\\
\end{tabular}
\end{table}

\subsection{Matrix optimization}
\subsubsection{Semidefinite representation for matrix geometric mean}
For $A\in\bS_{++}^n$, $B\in\bS_{+}^n$ and $\lambda\in[0,1]$, the $\lambda$-weighted geometric mean of $A$ and $B$ is defined by
\begin{equation*}
	G_{\lambda}(A,B)=A\#_{\lambda}B\coloneqq A^{\frac{1}{2}}\left(A^{-\frac{1}{2}}BA^{-\frac{1}{2}}\right)^{\lambda}A^{\frac{1}{2}}.
\end{equation*}
It can be shown that a second-order cone representation drawn from a successive minimum mediated sequence for a scalar geometric mean can be lifted to a semidefinite representation for the corresponding matrix geometric mean. That means, Algorithm~\ref{sec4:alg1} can be readily adapted to produce a semidefinite representation for the matrix geometric mean $G_{\lambda}$ when $\lambda$ is a rational number. 

\begin{theorem}\label{sec6:thm0}
Let $\lambda=\frac{q}{p}$ with $p,q\in\N^*,0<q<p$ and $(p,q)=1$.
Suppose that the scalar geometric mean inequality $x_1^{1-\lambda}x_2^{\lambda}\ge x_3$ admits a second-order cone representation drawn from a successive minimum mediated sequence: $x_{a}x_{b}\ge x_4^2$, $x_{i-1}x_{k_i}\ge x_{i}^2,i=5,\ldots,l$, where $a,b\in\{1,2,3\}, k_i\in\{1,2,\ldots,i-2\}$, $l=\left\lceil\log_2(p)\right\rceil+3$, and $x_l=x_3$. Then the matrix geometric mean $G_{\lambda}$ admits a semidefinite representation:
\begin{align*}
\{(X_1,X_2,T)\in\bS_{++}^n\times\bS_+^n\times\bS^n&\mid G_{\lambda}(X_1,X_2)\succeq T\}=\\
\bigg\{(X_1,X_2,T)\in\bS_{++}^n\times\bS_+^n\times\bS^n\,&\,\bigg|\,\,\exists \{X_{i}\}_{i=3}^l\subseteq\bS^n\hbox{\rm{ s.t. }}X_3=X_l\succeq T\hbox{\rm{ and}}\\
&\,\,\begin{bmatrix}X_{a}&X_{4}\\X_{4}&X_{b}\end{bmatrix}\succeq0, \begin{bmatrix}X_{i-1}&X_{i}\\X_{i}&X_{k_i}\end{bmatrix}\succeq0 \hbox{\rm{ for }} i=5,\ldots,l\bigg\}.
\end{align*}
\end{theorem}
\begin{proof}
By adapting the proof of \cite[Theorem 4.4]{sagnol2013}, one can show that the matrix geometric mean $G_{\lambda}$ satisfies the following extremal property:
\begin{align*}
    G_{\lambda}(X_1,X_2)=\max_{\succeq}\bigg\{X\in\bS_+^n\,&\,\bigg|\,\,\exists \{X_{i}\}_{i=3}^l\subseteq\bS^n\hbox{\rm{ s.t. }}X_3=X_l=X\hbox{\rm{ and}}\\
    &\,\,\begin{bmatrix}X_{a}&X_{4}\\X_{4}&X_{b}\end{bmatrix}\succeq0, \begin{bmatrix}X_{i-1}&X_{i}\\X_{i}&X_{k_i}\end{bmatrix}\succeq0 \hbox{\rm{ for }} i=5,\ldots,l\bigg\},
\end{align*}
where $\max_{\succeq}$ takes the largest element with respect to the Löwner ordering (i.e., $A\succeq B$ if and only if $A-B\succeq0$). The desired conclusion then easily follows.
\end{proof}

\begin{corollary}\label{sec6:thm1}
Let $\lambda=\frac{q}{p}$ with $p,q\in\N^*,0<q<p$ and $(p,q)=1$. Then $G_{\lambda}$ admits a semidefinite representation with $\left\lceil\log_2(p)\right\rceil$ linear matrix
inequalities of size $2n\times2n$ and one linear matrix
inequality of size $n\times n$.
\end{corollary}

\begin{remark}
As a comparison, the semidefinite representation for $G_{\lambda}$ provided in \cite{fawzi2017lieb,sagnol2013} needs as many as $2\left\lceil\log_2(p)\right\rceil-1$ linear matrix inequalities of size $2n\times2n$ and one linear matrix inequality of size $n\times n$.
\end{remark}

\subsubsection{Semidefinite representation for the multivariate generalization of Lieb's function}

Given $(\lambda_i)_{i=1}^m\in\Q_+^m$ with $\sum_{i=1}^m\lambda_i=1$, the multivariate generalization of Lieb's function is defined by
\begin{equation}\label{sec6:eq4}
	(A_1,\ldots,A_m)\in\bS_+^{n_1}\times\cdots\times\bS_+^{n_m}\mapsto A_1^{\lambda_1}\otimes\cdots\otimes A_m^{\lambda_m}\in\bS_+^{n_1\cdots n_m},
\end{equation}
where $\otimes$ denotes the Kronecker product of matrices.
We are able to give a semidefinite representation for \eqref{sec6:eq4} by iteratively using
\begin{equation*}
	A_1^{\lambda_1}\otimes\cdots\otimes A_m^{\lambda_m}\succeq T
	\iff \exists S\in\bS_+^{n_1n_2}\text{ s.t. }
	\begin{cases}
		A_1^{\frac{\lambda_1}{\lambda_1+\lambda_2}}\otimes A_2^{\frac{\lambda_2}{\lambda_1+\lambda_2}}\succeq S,\\
		S^{\lambda_1+\lambda_2}\otimes A_3^{\lambda_3}\otimes\cdots\otimes A_m^{\lambda_m}\succeq T,
	\end{cases}
\end{equation*}
and applying Corollary \ref{sec6:thm1} to
\begin{equation}
	A^{1-\lambda}\otimes B^{\lambda}=(A\otimes I)\#_{\lambda}(I\otimes B),
\end{equation}
where $I$ denotes the identity matrix of appropriate size.

Let us consider the following trace optimization problem:
\begin{equation}\label{sec6:eq5}
	\begin{cases}
		\sup\limits_{w_1,w_2,w_3} &\tr\left((w_1A_1+A_2)^{\lambda_1}\otimes(w_2A_3+A_4)^{\lambda_2}\otimes(w_3A_5+A_6)^{\lambda_3}\right)\\
		\,\,\,\,\,\,\text{s.t.} &w_1+w_2+w_3=1,\\
		&w_1,w_2,w_3\ge0,
	\end{cases}
\end{equation}
where $(\lambda_i)_{i=1}^3\in\Q_+^3$ with $\sum_{i=1}^3\lambda_i=1$, and $A_1,\ldots,A_6\in\bS_+^{n}$ are randomly generated positive definite matrices. By introducing a matrix variable $T$, this problem is equivalent to
\begin{equation}\label{sec6:eq6}
	\begin{cases}
		\sup\limits_{w_1,w_2,w_3} &\tr(T)\\
		\,\,\,\,\,\,\text{s.t.} &(w_1A_1+A_2)^{\lambda_1}\otimes(w_2A_3+A_4)^{\lambda_2}\otimes(w_3A_5+A_6)^{\lambda_3}\succeq T,\\
		&w_1+w_2+w_3=1,\\
		&w_1,w_2,w_3\ge0.
	\end{cases}
\end{equation}
We can further convert \eqref{sec6:eq6} into a semidefinite program (SDP) using the semidefinite representation for \eqref{sec6:eq4}. In Table \ref{tb4}, we present the results of solving \eqref{sec6:eq5} with different $(\lambda_1,\lambda_2,\lambda_3)$ and $n=2,3$\footnote{The script is available at \href{https://github.com/wangjie212/MiniSOC}{https://github.com/wangjie212/MiniSOC}.}. Here {\tt Mosek 9.0} serves as the SDP solver.

\begin{table}[htbp]
	\caption{Results for the trace optimization problem \eqref{sec6:eq4}. For each $n$, ``opt'' means the optimum, and ``time'' means the running time in seconds.}\label{tb4}
	\renewcommand\arraystretch{1.2}
	\centering
	\begin{tabular}{c|c|c||c|c}
		\multirow{2}{*}{$(\lambda_1,\lambda_2,\lambda_3)$}&\multicolumn{2}{c||}{$n=2$}&\multicolumn{2}{c}{$n=3$}\\
		\cline{2-5}
		&opt&time&opt&time\\
		\hline
		$(\frac{4}{9},\frac{3}{9},\frac{2}{9})$&4.656&0.03&10.46&2.38\\
		\hline
		$(\frac{8}{15},\frac{4}{15},\frac{3}{15})$&4.621&0.03&8.976&2.07\\
		\hline
		$(\frac{25}{44},\frac{10}{44},\frac{9}{44})$&5.054&0.05&9.474&3.13\\
		\hline
		$(\frac{37}{83},\frac{24}{83},\frac{22}{83})$&5.065&0.06&10.39&4.31\\
		\hline
		$(\frac{72}{169},\frac{47}{169},\frac{40}{169})$&4.579&0.07&10.03&5.57\\
		\hline
		$(\frac{159}{278},\frac{64}{278},\frac{55}{278})$&4.972&0.07&9.801&7.12\\
		\hline
		$(\frac{212}{453},\frac{133}{453},\frac{108}{453})$&4.039&0.07&9.188&5.92\\
		\hline
		$(\frac{391}{882},\frac{280}{882},\frac{211}{882})$&4.755&0.08&9.654&7.08\\
	\end{tabular}
\end{table}

\subsection{Quantum information}
Building on simple second-order cone representations for bivariate weighted geometric means, Fawzi and Saunderson provided semidefinite representations for several matrix functions arising from quantum information \cite{fawzi2017lieb}. We record their results below.

Suppose that $A\in\bS_+^n$, $B\in\bS_+^m$ are positive definite matrices and $\lambda$ is a rational number in $(0,1)$. The semidefinite representation for Lieb's function $F_{\lambda}(A,B)\coloneqq\tr(K^{\intercal}A^{1-\lambda}KB^{\lambda})$ ($K\in\R^{n\times m}$ are fixed) can be obtained via
\begin{equation}
	\tr(K^{\intercal}A^{1-\lambda}KB^{\lambda})\ge t \iff\exists T\in\bS^{nm}_+\text{ s.t. }
	\begin{cases}
		A^{1-\lambda}\otimes B^{\lambda}\succeq T,\\
		\vec(K)^{\intercal}T\vec(K)\ge t,
	\end{cases}
\end{equation}
where $\vec(K)$ is a column vector of size $nm$ obtained by concatenating the rows of $K$.
The semidefinite representation for the Tsallis entropy $S_{\lambda}(A)\coloneqq\frac{1}{\lambda}\tr(A^{1-\lambda}-A)$
can be obtained via
\begin{equation}
	\frac{1}{\lambda}\tr(A^{1-\lambda}-A)\ge t \iff\exists T\in\bS^{n}\text{ s.t. }
	\begin{cases}
		A\#_{\lambda} I\succeq T,\\
		\frac{1}{\lambda}\tr(T-A)\ge t,
	\end{cases}
\end{equation}
where $I$ is the identity matrix of appropriate size.
Assuming $n=m$, the semidefinite representation for the Tsallis relative entropy $S_{\lambda}(A\Vert B)\coloneqq\frac{1}{\lambda}\tr(A-A^{1-\lambda}B^{\lambda})$
can be obtained via
\begin{equation}
	\frac{1}{\lambda}\tr(A-A^{1-\lambda}B^{\lambda})\le t \iff\exists s\in\R\text{ s.t. }
	\begin{cases}
		\tr(A^{1-\lambda}B^{\lambda})\ge s,\\
		\frac{1}{\lambda}(\tr(A)-s)\le t.
	\end{cases}
\end{equation}
Notice that $\tr(A^{1-\lambda}B^{\lambda})$ is Lieb's function $F_{\lambda}(A,B)$ with $K\in\R^{n\times n}$ being the identity matrix.
\begin{remark}
Because of Corollary \ref{sec6:thm1}, we can construct semidefinite representations of smaller size for Lieb's function, the Tsallis entropy, the Tsallis relative entropy than the semidefinite representations given in \cite{fawzi2017lieb}.
\end{remark}

Now let us consider the following maximum entropy optimization problem (also tested in \cite{fawzi2017lieb}):
\begin{equation}\label{sec6:eq7}
	\begin{cases}
		\sup\limits_{w_i} &S_{\lambda}\left(\sum_{i=1}^mw_iA_i\right)\\
		\text{s.t.}&\sum_{i=1}^mw_i=1,\\
		&w_i\ge0,i=1,\ldots,m,
	\end{cases}
\end{equation}
where $A_1,\ldots,A_m\in\bS_+^n$ are fixed positive semidefinite matrices of trace one. In Table \ref{tb6} we present numerical results of solving \eqref{sec6:eq7} with $m=10$, $n=20,40$ and different $\lambda$\footnote{The script is available at \href{https://github.com/wangjie212/MiniSOC}{https://github.com/wangjie212/MiniSOC}.}. The results were obtained with the SDP solver {\tt Mosek 9.0}. For each instance, we use two algorithms to generate the required second-order cone representations: {\tt GreedyPowertwo} and the one proposed in the paper \cite{fawzi2017lieb}.
As we can see from the table, the approach with {\tt GreedyPowertwo} is more efficient than the one with the algorithm from \cite{fawzi2017lieb} by a factor $\sim2$.

\begin{table}[htbp]
	\caption{Results for the maximum entropy optimization problem \eqref{sec6:eq7}. The column labelled by ``opt'' records the optimum; the columns labelled by ``W'' and ``F-S'' record the running time in seconds of the approaches with {\tt GreedyPowertwo} and the algorithm from \cite{fawzi2017lieb}, respectively.}\label{tb6}
	\renewcommand\arraystretch{1.2}
	\centering
	\begin{tabular}{c|c|c|c||c|c|c}
		\multirow{2}{*}{$\lambda$}&\multicolumn{3}{c||}{$n=20$}&\multicolumn{3}{c}{$n=40$}\\
		\cline{2-7}
		&opt&W&F-S&opt&W&F-S\\
		\hline
		$\frac{5}{27}$&3.9230&0.87&1.84&5.2130&21.5&52.7\\
		\hline
		$\frac{12}{41}$&4.7037&0.86&1.51&6.5351&20.9&41.6\\
		\hline
		$\frac{34}{63}$&7.3642&1.01&1.73&11.530&29.3&39.8\\
		\hline
		$\frac{21}{107}$&4.0125&1.42&2.77&5.3271&32.9&80.4\\
		\hline
		$\frac{79}{168}$&6.4613&1.21&2.25&9.7798&33.3&60.6\\
		\hline
		$\frac{43}{213}$&4.0429&1.73&3.12&5.3920&49.3&91.0\\
		\hline
		$\frac{135}{422}$&4.9303&1.77&3.33&6.9332&49.0&97.1\\
		\hline
		$\frac{341}{745}$&6.3203&1.65&3.63&9.4727&45.1&98.9\\
	\end{tabular}
\end{table}

\section{Conclusion and discussion}
In this paper, we have studied the optimal size of simple second-order cone representations for a weighted geometric mean inequality and the minimum cardinality of mediated sets containing a given point. Fast heuristic algorithms have been proposed to compute approximately optimal simple second-order cone representations for weighted geometric mean inequalities. As (matrix) weighted geometric means widely appear in optimization, we believe that these results will lead to various applications and stimulate more research on this subject. We conclude the paper by listing some open problems for future research:

(1) Is there a polynomial time algorithm for computing an optimal simple second-order cone representation for a weighted geometric mean inequality? Or is this an NP-hard problem?

(2) We have proved $L(s_1,\ldots,s_m)\ge m-1$ in Theorem \ref{sec3:thm1}. On the other hand, we observed that for the tested examples, the lower bound $m-1$ is attained only if $\hat{s}$ is even. Is this always true? Can we prove $L(s_1,\ldots,s_m)\ge m$ given that $\hat{s}$ is odd?

(3) We have provided two lower bounds on $L(s_1,\ldots,s_m)$ in terms of either the dimension $m$ or the degree $\hat{s}$. Is it possible to prove an improved lower bound on $L(s_1,\ldots,s_m)$ combining $m$ and $\hat{s}$?

\section*{Acknowledgements}
The author would like to thank Chunming Yuan for helpful discussions and thank the referees for their insightful comments. This work was funded by NSFC-12201618.

\bibliographystyle{plain}
\bibliography{refer}

\end{document}